\documentclass[twoside,11pt]{article}
\usepackage[english]{babel}

\usepackage{amssymb,amsmath,mathrsfs,amsthm,indentfirst,graphicx,float,epstopdf}
\usepackage{cases}
\usepackage[numbers,sort&compress]{natbib}
\usepackage{subcaption}
\usepackage{enumerate}
\usepackage{enumitem} 
\usepackage{caption}
\usepackage{accents}
\allowdisplaybreaks  

\usepackage[dvipsnames]{xcolor}
\usepackage{tikz}
\usepackage{xxcolor}
\usepackage{tikz}
\usetikzlibrary{decorations.markings,backgrounds}
\usetikzlibrary{arrows.meta}
\usepackage{pgfplots}
\pgfplotsset{compat=1.17}
\usepackage[colorlinks=true]{hyperref}
\hypersetup{
	colorlinks,
	linkcolor={red!80!black},
	citecolor={blue!90!black},
	urlcolor={blue!80!black}
}
\usepackage{diagbox}

\usepackage{stmaryrd}

\newtheorem{The}{Theorem}[section]
\newtheorem{Lem}[The]{Lemma}
\newtheorem{Cor}[The]{Corollary}

\newtheorem{Pro}[The]{Proposition}

\theoremstyle{definition}
\newtheorem{Rem}{Remark}[section]

\graphicspath{{fig/}}

\numberwithin{equation}{section}
\numberwithin{figure}{section}

\begin{document}
	\captionsetup[figure]{labelfont={default},labelformat={default},labelsep=period,name={Fig.}}
	
	\title{\bf New Traffic Flow Model with Nonlinear Anticipation and Intelligent-control Boundary: Existence, Long-time Behavior and Large-relaxation-time Limit
    \vspace{-6mm}
		\footnotetext{\small
			*Corresponding author.}
		\footnotetext{\small E-mail addresses: zbhu@jxnu.edu.cn (Z. Hu), lisz@jxnu.edu.cn (S. Li), ming.mei@mcgill.ca (M. Mei), zejiawang@jxnu.edu.cn (Z. Wang).}}
	\author{{Zhibin Hu$^{1}$, Songzhi Li$^{1,\ast}$, Ming Mei$^{1,2,3}$, and Zejia Wang$^1$}\\[1mm]
		\small\it $^1$School of Mathematics and Statistics,  Jiangxi Normal University,\\
		\small\it    Nanchang, 330022, China \\
		\small\it $^2$Department of Mathematics, Champlain College Saint-Lambert,\\
		\small\it     Saint-Lambert, Quebec, J4P 3P2, Canada\\
		\small\it $^3$Department of Mathematics and Statistics, McGill University,\\
		\small\it     Montreal, Quebec, H3A 2K6, Canada
	}
	
	\date{}
	\maketitle
	
	\begin{quote}
		\small \textbf{Abstract}: In this paper, we propose a novel physical model for traffic flow incorporating nonlinear anticipation effects and an intelligent-control boundary, mathematically formulated as a damping boundary condition:
		\begin{align*}
			\begin{cases}
				u_t^\tau+v_x^\tau=0, & x\in(0,1),\; t>0,\\[1mm]
				v_t^\tau+g(u_x^\tau)u_x^\tau=\dfrac{f(u^\tau)-v^\tau}{\tau}, & x\in(0,1),\; t>0,\\[1mm]
				(u^\tau,v^\tau)(x,0)=(u_0^\tau(x),v_0^\tau(x)), & x\in(0,1),\\[1mm]
				u_x^\tau(0,t)=0,\quad u_x^\tau(1,t)=-ku_t^\tau(1,t), & k>0,\; t>0,
			\end{cases}
		\end{align*}
		where $f$ and $g$ are smooth functions satisfying suitable structural assumptions, and $\tau>0$ is the relaxation time. The primary objective is to rigorously investigate how the intelligent-control  boundary suppresses the stop-and-go phenomenon in the large-relaxation-time regime-a mechanism that has not been mathematically addressed in previous studies. Utilizing the energy method, we establish the global well-posedness and exponential time-decay of solutions for the original system under arbitrarily large initial data with small spatial derivatives. Furthermore, we analyze the asymptotic behavior in the large-relaxation-time limit $\tau\to\infty$, by introducing a novel technique that incorporates constant shifts into the initial data of the limiting system. This approach enables us to construct a modified auxiliary system, through which we successfully obtain the global convergence of the original solutions to the asymptotic profiles for all time $t$ as relaxation time $\tau\to\infty$. Numerical simulations further demonstrate that in the large-relaxation-time regime, the stop-and-go density waves emerging at the early stage are gradually suppressed by the damping boundary. This leads the traffic stream to eventually evolve into an essentially uniform profile, which perfectly validates our theoretical results.
		
		
		\indent \textbf{AMS (2020) Subject Classification}: 35L60, 35L65, 35B40
	\end{quote}

	\tableofcontents
	
	\thispagestyle{empty}

	\section{Introduction}
	{\bf New traffic flow model}. Macroscopic traffic flow models describe the spatio-temporal evolution of vehicle density $u(x,t)$ and vehicle flow rate $v(x,t)$ on a single road without sources or sinks. Their classification hinges on whether the flow is assumed to be instantaneously determined by the density. First-order (equilibrium) models close the conservation law by prescribing the flow as a function of the density alone, $v=f(u)$. The prototypical example is the Lighthill-Whitham-Richards (LWR) model \cite{LW-1955,R-1956}, which represents the starting point for the mathematical modeling of vehicular flows. Second-order (non--equilibrium) models relax the instantaneous equilibrium assumption by treating the velocity (or the flow) as an independent dynamic variable coupled to the density through an additional partial differential equation. Prominent examples include the Payne-Whitham model \cite{P-1971,Whitham-1974} and the Aw-Rascle-Zhang model \cite{AR-2000,Z-2002}.
	
	In the present work, we adopt the viewpoint that the flow rate $v$ itself---rather than the velocity---is the natural non-equilibrium variable, and we postulate an evolution equation that captures the physical mechanisms governing its rate of change.
	
	\begin{itemize}
		\item \textbf{Relaxation mechanism.} In a spatially uniform traffic stream ($u_x=0$), drivers only react to the local density. Due to finite reaction and vehicle adaptation times, the actual flow cannot instantaneously match the equilibrium value $f(u)$. Instead, it relaxes toward it at a rate proportional to the current deviation:
		\[
		\left.\frac{\partial v}{\partial t}\right|_{\text{relax}} = \frac{f(u)-v}{\tau},
		\]
		where $\tau>0$ is the relaxation time constant. Physically, the relaxation time implies that, when subjected to acceleration/deceleration disturbances (e.g., braking vehicles ahead, traffic congestion, road bottlenecks), drivers cannot instantly adjust to the target speed; driving response inertia exists. This mechanism can be derived from microscopic car-following models in the homogeneous limit. In the optimal velocity model \cite{BHNSS-1995}, the acceleration of the $n$-th vehicle satisfies $\dot{v}_n = \tau^{-1}[V(\Delta x_n) - v_n]$, where $V$ is the optimal velocity function. Under uniform spacing and velocity, the headway $\Delta x_n$ is constant and related to the macroscopic density by  $\Delta x \approx 1/u$. Passing to the continuum limit and identifying the equilibrium flow as $f(u) = u \cdot V(1/u)$, the microscopic dynamics are exactly reduced to the above relaxation law. The time scale $\tau$ thus encapsulates the aggregate effect of perception delay, decision-making, and vehicle dynamics.
		
		\item \textbf{Anticipatory response.} When the density is non-uniform ($u_x\neq0$), drivers observe the downstream profile and proactively adjust their speed in addition to the relaxation process: an increasing density ahead $u_x>0$ induces a deceleration tendency, whereas a decreasing density $u_x<0$ encourages acceleration. This anticipatory behavior is modeled as an additional contribution to the rate of change of the flow:
		\[
		\left.\frac{\partial v}{\partial t}\right|_{\text{anticip}} = -g(u_x)u_x,
		\]
		where $g(\cdot)>0$ is a sensitivity function. When $g$ is constant, this term reduces to the gradient correction found in the classical Payne--Whitham model, which can be obtained from the optimal velocity model by expanding the headway to first order \cite{BHNSS-1995,P-1971,Whitham-1974}. The present work generalizes this by allowing $g$ to depend on $u_x$, motivated by the physical consideration that drivers may react more or less strongly depending on the steepness of the density gradient.
	\end{itemize}
	
	The conservation of vehicles provides the first governing equation
	\begin{align}\label{eq:cons}
		u_t + v_x = 0.
	\end{align}
	Assuming that the relaxation mechanism and the anticipatory response act simultaneously and their contributions superpose, the total rate of change of the flow is as follows,
	\[
	\frac{\partial v}{\partial t} = \frac{f(u)-v}{\tau} - g(u_x)u_x,
	\]
	which yields the second governing equation,
	\begin{align}\label{eq:second}
		v_t + g(u_x)u_x = \frac{f(u)-v}{\tau}.
	\end{align}
	This superposition hypothesis constitutes the central modeling postulate. When the gradient sensitivity is constant, $g(u_x)\equiv c_0^2>0$, the system reduces to the Payne--Whitham model \cite{P-1971,Whitham-1974} with the convection term neglected. The present work extends this structure by allowing the gradient sensitivity to depend on the gradient itself, i.e., $g=g(u_x)$. This generalization provides a flexible phenomenological framework that can capture more complex anticipatory behaviors while preserving the clear separation of roles between the relaxation time $\tau$ and the gradient sensitivity $g$.
	
	In summary, the focus of the present paper is to investigate the following conservation laws involving relaxation effect on the interval $x\in(0,1)$ with $t>0$:
	\begin{align}\label{a}
		\begin{cases}
			u^{\tau}_t+v^{\tau}_x=0,\\[1mm]
			v^{\tau}_t+g(u^{\tau}_x)u^{\tau}  _x=\frac{f\left( u^{\tau} \right) -v^{\tau}}{\tau},
		\end{cases}
		\quad x\in (0,1),t>0,
	\end{align}
	where $u^\tau(x,t)$ and $v^\tau(x,t)$ are the vehicle density and vehicle flow rate, respectively, the unknown functions of   $x$ and $t$ with respect to the parameter of relaxation time $\tau>0$, and $f(\cdot)$ is a  smooth function  and $g(\cdot)>0$ is a prescribed sensitivity function.

	To motivate the boundary conditions, we consider first the baseline case of reflecting ends, $u_x=0$ at both boundaries. An energy computation (cf. \eqref{1q1} below) shows that, in this case, density waves are barely dissipated when the relaxation time is large, so that stop-and-go patterns (cf. \cite{Fayolle-2026}) persist for a long time. To suppress this phenomenon, we impose the so-called intelligent-control bounaary $u_x(1,t)=-ku_t(1,t)$ at the exit, namely, the damping boundary, which keeps extracting energy from the system at the boundary, together with the Neumann condition $u_x(0,t)=0$ at the entrance. In summary, we study \eqref{a} subject to the initial data and the new proposed physical boundary conditions given by
	\begin{align}\label{b}
		\begin{cases}
			\left( u^{\tau},v^{\tau} \right) (x,0) =\left( u_0^\tau(x) ,v_0^\tau(x) \right) ,&		x\in (0,1),\\[1mm]
			u^{\tau}_x(0,t)=0,u^{\tau}_x(1,t)=-ku^{\tau}_t(1,t), &k>0, t>0,
		\end{cases}
	\end{align}
	where the boundary conditions presented in the above are the Neumann boundary at $x=0$, and the damping boundary at $x=1$, which   
	admit a clear physical interpretation in terms of on-ramp and off-ramp operations:
	\begin{itemize}
		\item {\bf Neumann boundary ($u_x=0$)} at the left boundary $x=0$. This enforces a spatially uniform density profile at the entrance. This corresponds to a steady, undisturbed inflow from an upstream on-ramp that feeds vehicles into the domain at a constant rate, without creating local accumulations or depletions.
		
		\item {\bf Intelligent-control boundary ($u_x=-ku_t$)} at the right boundary $x=1$. This links the spatial slope of the density to its local rate of change, thereby implementing an adaptive outflow regulation. When the exit density begins to rise ($u_t>0$), indicating that vehicles are starting to accumulate near the exit, the boundary induces a negative gradient $u_x<0$, meaning the density decreases further downstream. This gradient helps to draw vehicles out of the domain and relieves the emerging queue. Conversely, when the exit density is falling ($u_t < 0$), suggesting that vehicles are dispersing, a positive gradient $u_x>0$ is produced, which slightly increases the downstream density and prevents the mainline from emptying too quickly. The positive parameter $k$ measures the sensitivity of this regulation: a larger $k$ corresponds to a stronger response to density changes. In essence, the right boundary functions as an intelligent off-ramp controller:  it senses the local density trend in real time and adjusts the outflow in the opposite direction---releasing more vehicles when the density begins to accumulate, and gently restricting the outflow when the density tends to drop. 
	\end{itemize}
	
	Throughout this paper, we assume that the initial data satisfy the
	following compatibility conditions of order one. Let
	\begin{align}\label{comp-1}
h(s):=g'(s)s+g(s), \quad
		u^{(1)}:=-v_{0x}^{\tau},\quad
		u^{(2)}:=h(u_{0x}^{\tau})u_{0xx}^{\tau}
		+\frac{1}{\tau}\Big(v_{0x}^{\tau}-f'(u_0^{\tau})u_{0x}^{\tau}\Big),
	\end{align}
	so that $u^{(1)}=\partial_tu^\tau(\cdot,0)$ and
	$u^{(2)}=\partial_t^2u^\tau(\cdot,0)$. We assume
	\begin{align}\label{comp-2}
		u^{(0)}_{x}(0)=0,\quad u^{(1)}_{x}(0)=0,\quad
		u^{(0)}_{x}(1)=-ku^{(1)}(1),\quad u^{(1)}_{x}(1)=-ku^{(2)}(1),
	\end{align}
	where $u^{(0)}:=u_0^\tau$; equivalently, in terms of the original data,
	\[
	u_{0x}^{\tau}(0)=0,\quad v_{0xx}^{\tau}(0)=0,\quad
	u_{0x}^{\tau}(1)=kv_{0x}^{\tau}(1),\quad
	v_{0xx}^{\tau}(1)=ku^{(2)}(1).
	\]
	
	Moreover, we assume that the smooth functions $f(s)$ and $g(s)$ satisfy the following hypotheses:
	\begin{enumerate}[label=(H\textsubscript{\arabic*}),ref=H\textsubscript{\arabic*}]
		\item \label{h1} There exists a constant $\delta>0$ such that $g(s)\ge \delta$ for all $s\in\mathbb{R}$.
		\item \label{h2} $g'(s)s\ge 0$ for all $s\in\mathbb{R}$.
		\item \label{h3} $f(0)=0$.
	\end{enumerate}
	Typical examples include $g(s)=|s|^p+c$ with $c>0,\ p>0$, or $g(s)=e^{s^2}$. Assumption \eqref{h3} is physically natural, meaning that the equilibrium flow vanishes when the vehicle density is zero.
	
	{\bf Background of study}. Formally, the system \eqref{eq:cons}--\eqref{eq:second} belongs to the class of relaxation models of Jin--Xin type \cite{Jin-1995}. It couples a conservation law with a rate equation that describes the relaxation of a non-equilibrium variable toward an equilibrium manifold, a structure that has been extensively studied in the context of hyperbolic systems with relaxation. This connection enables us to leverage existing frameworks on the well-posedness, stability, and relaxation limit of such systems, and it places the present model within a broader mathematical framework that has already proven fruitful for the analysis of traffic flow and other continuum problems such as  river flows, chromatography, and kinetic theory \cite{Charke-1978,Chen-1995,Whitham-1974}. In 1987, T.-P. Liu \cite{Liu-1987} first established the fundamental theoretical framework for relaxation hyperbolic systems. Based on his work, Jin and Xin \cite{Jin-1995} proposed the following hyperbolic system with relaxation as follows:
	\begin{align}\label{1.1}
		\begin{cases}
			u^{\tau}_t+v^{\tau}_x=0,\\[2pt]
			v^{\tau}_t+au^{\tau}  _x=\frac{f\left( u^{\tau} \right) -v^{\tau}}{\tau},
		\end{cases}
		\quad \text{for}~\tau >0.
	\end{align}
	
	Regarding the well-posedness, the relaxation-time-limit, the existence and stability of nonlinear waves for the Cauchy problem or initial-boundary value problem of the Jin-Xin relaxation model  has been extensively studied. Some studies focus on the global existence of weak solutions or BV solutions \cite{Na-1996,Luo-2000,Amadori-2001,Zeng-2011}. The zero-relaxation-time limit of the system \eqref{1.1} was deeply studied by Chen-Levermore-Liu \cite{Chen-1994} and Chen-Liu \cite{Chen-1993}  and  Lattanzio-Pierangelo \cite{Lattanzio-1994} and Wang-Xin \cite{Wang-1998} and others \cite{Lattanzio-1997,Beltran-2026}.  For more relevant models on the zero-relaxation-time limit, we refer the reader to \cite{Feng-2020,Feng-2023,Peng-2011,Qiu-2002}, and for the large-relaxation-time limit, to \cite{Feng-2024-01,Feng-2024-02}. The numerical analysis and computations for relaxation time limits were systematically carried out by Jin and Xin \cite{Jin-1995}. Moreover, a series of studies were concerned with the existence for hyperbolic conservation laws with relaxation effects, such as the sharp viscous shock waves \cite{X-M-S-W-2025,H-M-W-2026}.  In addition, we refer to \cite{L-W-Y-1997,L-W-Y-1998,MC-1996,Mei-1998,Liu-2003,Zhang-2004} for the stability of viscous shock waves for the Jin-Xin system.  The convergence to the traveling wave solutions, as the relaxation time goes to zero, was recently considered by Jin and Liu \cite{Jin-1998}.  On the other hand, the asymptotic limit of relaxation time with a boundary effect was given in \cite{Mei-1999,Yang-2000,Hsiao-2000,Liu-2001,Hsiao-2001}. Regarding the stability of travelling waves to the IBVP for other models of hyperbolic conservation laws, we refer the reader to the interesting works in \cite{Li-2000,Liu-19971,Liu-19972,Matsumura-1999}.
	
	{\bf Main purpose}. In most of the above-mentioned works, the relaxation-time limits were studied only in the regime $\tau\to 0^+$. This means that drivers adjust their speed more rapidly, nearly reaching the equilibrium state instantaneously. 
	
	On the other hand, in the large-relaxation-time regime ($\tau\gg 1$), drivers adjust their speeds sluggishly, which renders the traffic flow prone to instability and to persistent stop-and-go density waves. Under reflecting boundary conditions ($u_x=0$ at both ends), such waves persist indefinitely. In contrast, the intelligent-control boundary  $u_x(1,t)=-ku_t(1,t)$ is specifically designed to suppress this undesirable phenomenon. However, to the best of our knowledge, intelligent‑control boundary conditions of this type---mathematically damping boundary conditions---have not yet been studied in this setting. In this paper, we focus on this regime. We establish the global well-posedness and exponential decay of solutions, assuming large initial data but small derivatives (Theorem~\ref{thm1}). More importantly, we investigate the large-relaxation-time limit: as $\tau\to\infty$, the solution converges uniformly in time to the shifted asymptotic profile of the limiting system, with a rate of $O(\tau^{-1})$ (Theorem~\ref{thm3}). Moreover, our analysis reveal that, even in this large-relaxation-time regime, the intelligent-control boundary keeps driving the traffic stream
	toward a spatially uniform equilibrium as time goes on, so that the stop-and-go phenomenon eventually disappears; the numerical simulations presented in Section \ref{sec5} illustrate this behavior.
	
	Formally taking the limit as $\tau\to\infty$ in \eqref{a}, and denoting the limiting functions by $(\bar{u},\bar{v})$, we obtain the following system for the asymptotic profile $(\bar{u},\bar{v})$:
	\begin{align}\label{2a}
		\begin{cases}
			\bar{u}_t+\bar{v}_x=0,\\[1mm]
			\bar{v}_t+g(\bar{u}_x)\bar{u} _x=0,
		\end{cases}
		\quad x\in (0,1),\; t>0,
	\end{align}
	subject to the initial-boundary conditions
	\begin{align}\label{bb}
		\begin{cases}
			(\bar{u},\bar{v})(x,0)=(\bar{u}_0(x),\bar{v}_0(x)), &x\in (0,1),\\[1mm]
			\bar{u}_{x}(0,t)=0,\qquad \bar{u}_{x}(1,t)=-k\bar{u}_t(1,t), &k>0,\; t>0.
		\end{cases}
	\end{align}
	We expect that the solutions $(u^\tau(x,t),v^\tau(x,t))$ of the original relaxation model \eqref{a} and \eqref{b} will converge to the corresponding asymptotic profile $(\bar{u}(x,t),\bar{v}(x,t))$ of the system \eqref{2a} and \eqref{bb} as $\tau\to\infty$ in the form of
	\[\| (u^{\tau}-\hat{u},v^{\tau}-\hat{v})(\cdot,t)\|_{L^{\infty}(0,1)}\le C\tau ^{-1},
	\]
	once the initial perturbation decays like:
	\[
	\|u_{0}^\tau-\hat{u}_0 \|_{H^1(0,1)}+\|v_{0}^\tau-\hat{v}_0\|_{H^1(0,1)}\leq C\tau^{-1}.
	\]
	Here,  the solutions $(u^\tau(x,t),v^\tau(x,t))$ are proved to  exist globally in time:
	\[\|(u^{\tau}-u^\tau_*,v^\tau-v^\tau_*)( \cdot,t)\|_{L^2(0,1)}\leq Ce^{-\sigma t},\]
	for some constant $\sigma>0$, where the constant pair  $(u^\tau_*,v^\tau_*)$ satisfies:
	\[
	u^\tau_*:=\int_0^1 u_0^\tau(x) dx + \int^\infty_0 [v^\tau(0,t)-v^\tau(1,t)]dt, \ \ \ v^\tau_*:=f(u^\tau_*);
	\]
	and   the solution $(\bar{u}(x,t),\bar{v}(x,t))$ is proved to  exist globally in time:
	\[\|(\bar{u}-\bar{u}_*,\bar{v}-\bar{v}_*)( \cdot,t)\|_{L^2(0,1)}\leq Ce^{-\gamma t},\]
	for some constant $\gamma>0$, where the constant pair $(\bar{u}_*,\bar{v}_*)$ satisfies:
	\[
	\bar{u}_*=\lim_{t\to\infty}\int_0^1 \bar{u}(x,t) dx, \ \ \ \ \bar{v}_*=\lim_{t\to\infty}\int_0^1 \bar{v}(x,t) dx,
	\]
	respectively.
	
	\textbf{Difficulties and strategies.}	However, when the relaxation time $\tau$ is sufficiently large, the dissipative mechanism is extremely weak or even disappears, which makes the study analytically challenging, and fundamentally different from the previous works. 
	Let us outline the primary difficulties of the study and  our strategies in the proof. 
	
	Firstly, the energy method for the original system \eqref{a}--\eqref{b} faces a fundamental difficulty: the boundary conditions in \eqref{b} prevent the use of the Poincar\'e inequality. As a result, the standard energy framework for wave equations-which controls only derivatives of the solution-fails, since it contains no term involving $u^\tau$ itself. However, the nonlinear source term $\frac{f(u^\tau)}{\tau}$ requires precisely an a priori bound on $u^\tau$ to close the estimate. To overcome this obstacle, we impose an {\it a priori} assumption that the spatial and temporal derivatives decay exponentially; from this we derive the missing bound on $u^\tau$. This not only closes the energy estimates but also yields, as a direct byproduct, the long-time asymptotic behavior of the solution.
	
	The second main difficulty is to establish the global convergence of $(u^\tau,v^\tau)$ for all time $t$ in the large-relaxation-time limit $\tau\to\infty$. Formally sending $\tau\to\infty$ in \eqref{a} yields the limit system \eqref{2a}--\eqref{bb}, which enjoys translational invariance: any constant shift of a solution, with the initial data shifted accordingly, remains a solution. This invariance introduces an ambiguity: the limit system alone does not determine which shifted solution serves as the asymptotic profile of $(u^\tau,v^\tau)$.
	
	Our key idea is to remove this ambiguity by choosing the shifts precisely as the limiting constants $\bar{u}_*$ and $\bar{v}_*$ provided by Theorem \ref{thm2}, specifically, we set 
	$(\hat{u},\hat{v}):=(\bar{u}-\bar{u}_*,\bar{v}-\bar{v}_*)$, and $(\hat{u}_0,\hat{v}_0):=(\bar{u}_0-\bar{u}_*,\bar{v}_0-\bar{v}_*)$.

	By Theorem \ref{thm2}, $(\hat{u},\hat{v})$ decays exponentially in time to zero. This seemingly simple translation is the central technical innovation: it aligns the asymptotic profile with the actual limit of the original system. Moreover, combined with the physical condition (\ref{h3}), the exponential decay of $\hat{u}$ implies that $f(\hat{u})$ also decays exponentially, which is precisely the property needed to control the nonlinear source term $\frac{f(u^\tau)-v^\tau}{\tau}$ and close the convergence estimates. Thus, the translational invariance, which initially poses an ambiguity, is turned into a powerful tool that enables the full convergence proof.
	
	{\bf Notations.} In this paper, the symbol $C$ is a generic positive constant that may vary between lines.  
	\begin{itemize}
		\item $ L^2(0,1)$ denotes the space of measurable functions on $(0,1)$ which are square integrable, with the norm
		$$
		\|f\|:=\left(\int_0^1|f(x)|^2 d x\right)^{1 / 2}.
		$$
		\item $H^k(0,1)(k\geq 0)$ is the $k$-th order Sobolev space on $(0,1)$ with the norm $$\|f\|_{k}:=\left(\sum_{j=0}^k\left\|\partial_x^j f\right\|_{L^2}^2\right)^{1 / 2}.$$ 
	\end{itemize}
	
	Let $T>0$ be a constant and $B$ be a Banach space, respectively. We denote by $C^k([0, T];B)$ $(k\geqq 0)$ the space of $B$-valued $k$-times continuously differentiable functions on $[0, T]$, and by $L^2([0, T] ; B)$  the space of $B$-valued $L^2$-functions on $[0, T]$. The corresponding spaces of $B$-valued functions on $[0, \infty)$ are defined similarly. We also denote that
	\[
	C^k([0,T];B)\times C^k([0,T];B) =: [C^k([0,T];B)]^2.
	\]
	
	The rest of this paper is organized as follows. In Section \ref{sec2}, we state our main results. Section \ref{sec3} is devoted to the energy estimates and the proof of global existence and long-time behavior for the original system \eqref{a}--\eqref{b}, thereby completing the proof of Theorem \ref{thm1}. In Section \ref{sec4}, we investigate the large-relaxation-time limit of system \eqref{a}--\eqref{b} and prove Theorem \ref{thm3}.  In Section \ref{sec5}, we present numerical simulations to validate the theoretical results obtained in Theorem \ref{thm3}. Finally, the proof of Theorem \ref{thm2}, concerning the global existence and long-time behavior for the quasilinear wave equation system \eqref{2a}--\eqref{bb}, is provided in the Appendix.
	
	\section{Main Results}\label{sec2}
	
	In this section, we shall present the main results of this paper.   We begin with the global existence and long-time behavior for the original problem \eqref{a}--\eqref{b}.
	\begin{The}[Existence and long-time behavior for the original system]\label{thm1}
		Assume that $(u_0^\tau,v_0^\tau)\in \left(H^3(0,1)\right)^2$  satisfies the compatibility conditions \eqref{comp-1}--\eqref{comp-2} and $g$ satisfies \eqref{h1}--\eqref{h2}. There exists a positive constant $\varepsilon_1$ such that if 
		$$\ \| u_{0x}^\tau \|_{2} + \| v_{0x}^\tau\|_{2}\leq\varepsilon_1,$$
		then there exists a positive constant $\tau^*$ independent of $t$ and $\tau$ such that, when $\tau \geq \tau^*$, the problem \eqref{a} with \eqref{b} admits a unique global solution $(u^\tau,v^\tau)\in \left(C^0\left([0,+\infty);H^3(0,1)\right)\right)^2$ satisfying 
		\begin{align}\label{tauest}
			\|u^{\tau}_x( \cdot,t)\|_2+\|v^{\tau}_x( \cdot,t)\|_2\leq C\left(\| u_{0x}^\tau \|_{2} + \| v_{0x}^\tau\|_{2} \right)e^{-\sigma t},
		\end{align}
		and 
		\begin{align}\label{tauest2}
			\|(u^{\tau}-u^\tau_*)( \cdot,t)\|\leq C\left(\| u_{0x}^\tau \| + \| v_{0x}^\tau\|\right)e^{-\sigma t},
		\end{align}
		with
		\begin{align}\label{tauest3}
			\|(v^{\tau}-v^\tau_*)( \cdot,t)\|\leq
			\left\{\begin{aligned}
				&C\frac{1+\tau}{|1-\tau\sigma|}\left(\|v_0^\tau\|_1+\|u_0^\tau\|_1+|f(0)|\right)e^{-\min\{\sigma,\frac{1}{\tau}\}t}, \quad &\text{ if }& \tau\neq\frac{1}{\sigma},\\
				&C\frac{1+\sigma}{\sigma-\sigma'}\left(\|v_0^\tau\|_1+\|u_0^\tau\|_1+|f(0)|\right)e^{-\sigma't}, &\text{ if }& \tau=\frac{1}{\sigma},
			\end{aligned}\right.
		\end{align}
		for any $0<\sigma'<\sigma$, where $C,\sigma$ are positive constants independent of $t$ and $\tau$, and  $(u^\tau_*, v^\tau_*)$ is the constant pair  given by 
		\begin{equation}\label{new1}
			u^\tau_*:=\int_0^1 u_0^\tau(x) dx + \int^\infty_0 [v^\tau(0,t)-v^\tau(1,t)]dt, \quad v^\tau_*:=f(u^\tau_*).
		\end{equation}
	\end{The}
	
	\begin{Rem}
		It is worth noting that condition (\ref{h3}) is not required in Theorem \ref{thm1}.
	\end{Rem}
	
	\begin{Rem}
		As will be seen in the next section, the decay rates presented in
		\eqref{tauest2} and \eqref{tauest3} are determined by
		\[
		\lambda=\min\left\{\frac{2k\delta}{1+k^2\delta},\frac{\sqrt{\delta}}{2}\right\},
		\qquad
		\sigma\in\left(0,\frac{\lambda}{8\left(1+\frac{\lambda}{\sqrt{\delta}}\right)}\right),
		\]
		where $\delta>0$ is given by \eqref{h1} and $k>0$ is the damping
		coefficient in \eqref{b}. Note that, for $k$ sufficiently small,
		\[
		\lambda=\frac{2k\delta}{1+k^2\delta}=O(k)\to0,
		\]
		and hence $\sigma\to0$ as $k\to0$. Thus the exponential decay is
		driven by the damping boundary condition \eqref{b}: as the boundary
		control is switched off ($k\to0$), the decay rate formally vanishes.
		This reveals the dependence of the long-time behavior of the traffic
		flow on the damping boundary.
	\end{Rem}
	
	\begin{Rem}
		We emphasize that the smallness assumption in Theorem \ref{thm1} is imposed only on the spatial derivatives of the initial data, namely \(\|u^\tau_{0x}\|_2+\|v^\tau_{0x}\|_2\le \varepsilon_1\), rather than on the initial data themselves. Thus, the initial profiles may be arbitrarily large in amplitude, provided they are sufficiently flat (i.e., their spatial variations are small).
	\end{Rem}

	We now state the existence and decay property of the solution $(\bar{u},\bar{v})$ to the initial-boundary value problem \eqref{2a}--\eqref{bb}. Although the global existence in time of solutions to the quasilinear wave equations  was well studied by Greenberg and Li \cite{Greenberg-1984} using Riemann invariants, the  explicit decay estimates were not provided in their work. In this paper, we revisit the same problem via the energy method and obtain both global existence and explicit exponential decay rates. The proof of the following theorem will be given in the Appendix.
	
	\begin{The}[Existence and long-time behavior for the profile system]\label{thm2} 
		Assume that $(\bar{u}_0,\bar{v}_0)\in \left(H^3(0,1)\right)^2$ satisfies compatibility conditions and $g$  satisfies \eqref{h1}--\eqref{h2}. There exists a positive constant $\varepsilon_2$ such that if 
		\[ \|\bar{u}_{0x}\|_2+\|\bar{v}_{0x}\|_2\leq \varepsilon_2, \]
		then the problem \eqref{2a}--\eqref{bb} admits a unique global solution $(\bar{u},\bar{v})\in\left(C^0\left([0,+\infty);H^3(0,1)\right)\right)^2$ satisfying
		\begin{align}
			\|\bar{u}_x( \cdot,t)\|_2+\|\bar{v}_x( \cdot,t)\|_2\leq C\left(\|\bar{u}_{0x}\|_2+\|\bar{v}_{0x}\|_2\right)e^{-\gamma t},
		\end{align}
		and
		\begin{align}\label{bar u*}
			\| (\bar{u}-\bar{u}_*)( \cdot,t)\|+\|(\bar{v}-\bar{v}_*)( \cdot,t)\|\leq C\left( \|\bar{u}_{0x}\|+\|\bar{v}_{0x} \|\right)e^{-\gamma t},
		\end{align}
		where $C,\gamma$ are positive constants independent of $t$, and   the constant pair $(\bar{u}_*,\bar{v}_*)$ is given by
		\begin{equation}
			\bar{u}_*=\lim_{t\to\infty}\int_0^1 \bar{u}(x,t) dx, \ \ \ \ \bar{v}_*=\lim_{t\to\infty}\int_0^1 \bar{v}(x,t) dx.
		\end{equation}
	\end{The}

	By Theorem \ref{thm2}, we now exploit the translation invariance of the limiting system \eqref{2a} to construct suitable asymptotic profiles for the original solution as $\tau\to\infty$. The invariance property can be stated as follows: if $(\bar{u},\bar{v})$ is a solution of \eqref{2a}--\eqref{bb} with initial data $(\bar{u}_0,\bar{v}_0)$, then for any constants $h_1,h_2\in\mathbb{R}$, the translated pair $(\bar{u}+h_1,\bar{v}+h_2)$ is again a solution of the same system, provided that the initial data are shifted correspondingly to $(\bar{u}_0+h_1,\bar{v}_0+h_2)$. This observation suggests that the asymptotic profile of $(u^\tau,v^\tau)$ as $\tau\to\infty$ should be obtained by a suitable vertical translation of the limiting solution $(\bar{u},\bar{v})$, with independent shifts for each component.
	
	Guided by this, we define the profiles $(\hat{u},\hat{v})$ by subtracting the  constants $\bar{u}_*,\bar{v}_*$ from the solution $(\bar{u},\bar{v})$. Specifically, set
	\[
	\hat{u}_0(x):=\bar{u}_0(x)-\bar{u}_*,\quad 
	\hat{v}_0(x):=\bar{v}_0(x)-\bar{v}_*,
	\]
	and let $(\hat{u},\hat{v})$ be the unique global solution of \eqref{2a}--\eqref{bb} corresponding to the initial data $(\hat{u}_0,\hat{v}_0)$. Equivalently, $(\hat{u},\hat{v})=(\bar{u}-\bar{u}_*,\bar{v}-\bar{v}_*)$ satisfies the system
	\begin{align}\label{hatuv-equa}
		\begin{cases}
			\hat{u}_t+\hat{v}_x=0, & x\in(0,1),\; t>0,\\[1mm]
			\hat{v}_t+g(\hat{u}_x)\hat{u}_x=0, & x\in(0,1),\; t>0,\\[1mm]
			(\hat{u},\hat{v})(x,0)=(\hat{u}_0(x),\hat{v}_0(x)), & x\in(0,1),\\[1mm]
			\hat{u}_x(0,t)=0, \hat{u}_x(1,t)=-k\hat{u}_t(1,t), &k>0, t>0.
		\end{cases}
	\end{align}
	
	The following corollary summarizes the decay properties of these profiles, which follow directly from Theorem \ref{thm2}.
	
	\begin{Cor}\label{corhat}
		Under the assumptions of Theorem \ref{thm2}, let $\bar{u}_*,\bar{v}_*$ be the constants obtained therein, and define $(\hat{u},\hat{v})$ as above. Then there exist positive constants $C,\gamma$ independent of $t$ such that
		\begin{align}\label{uvhat1}
			\|\hat{u}_x( \cdot,t)\|_2+\|\hat{v}_x( \cdot,t)\|_2 &\le C\left(\|\hat{u}_{0x}\|_2+\|\hat{v}_{0x}\|_2\right)e^{-\gamma t},
		\end{align}
		and 
		\begin{align}\label{uvhat2}
			\|\hat{u}( \cdot,t)\|+\|\hat{v}( \cdot,t)\| &\le C\left(\|\hat{u}_{0x}\|+\|\hat{v}_{0x}\|\right)e^{-\gamma t}.
		\end{align}
	\end{Cor}

	Finally, we state our main convergence result, which establishes the convergence of the original solutions to the asymptotic profiles in the large-relaxation-time limit.
	
	\begin{The}[Large-relaxation-time limit]\label{thm3}
		Assume that $g$ satisfies \eqref{h1}--\eqref{h2} and $f$ satisfies \eqref{h3}. Then there exists a positive constant $\varepsilon_3\leq \min\{\varepsilon_1,\varepsilon_2\}$ such that, when
		\[\|(u_{0x}^\tau,\hat{u}_{0x})\|_{2} + \| (v_{0x}^\tau,\hat{v}_{0x})\|_{2}\leq \varepsilon_3,\]
		the solutions $(u^\tau,v^\tau)$ of \eqref{a}--\eqref{b} satisfy the
		following uniform-in-time estimate with respect to the
		asymptotic profile $(\hat u,\hat v)$ of \eqref{hatuv-equa}:
		\begin{equation}\label{goal1}
			\| (u^{\tau}-\hat{u},v^{\tau}-\hat{v})( \cdot,t) \| _{L^{\infty}}\le C\left( \| u_{0}^{\tau}-\hat{u}_0\| _1+\| v_{0}^{\tau}-\hat{v}_0\| _1 \right) +C\tau ^{-1}, 
		\end{equation}
		uniformly for all time  $t$,  as $\tau\ge \tau^*$, where $\tau^*$ is defined in Theorem \ref{thm1}.
		
		Moreover, if the initial data satisfy the additional condition
		\[
		\|u_{0}^\tau-\hat{u}_0 \|_1+\|v_{0}^\tau-\hat{v}_0\|_1\leq C\tau^{-1},
		\]
		then the solutions of the original system \eqref{a}--\eqref{b} converge to those of the limiting system \eqref{hatuv-equa} as $\tau\to+\infty$, in the sense that
		\begin{equation}\label{goal2}
			\| (u^{\tau}-\hat{u},v^{\tau}-\hat{v})(\cdot,t)\|_{L^{\infty}}\le C\tau ^{-1}, \ \ \ \mbox{ uniformly for all time } t.
		\end{equation}
	\end{The}
	
	\begin{Rem} We now discuss some technical issues regarding the convergence as $\tau\to\infty$. In fact, 
		we emphasize that the introduction of the shifted profiles \((\hat{u},\hat{v})\) is essential for obtaining global-in-time convergence as \(\tau\to\infty\). Indeed, if one naively takes the unshifted profiles \((\bar{u},\bar{v})\) as the asymptotic limit, the residual term in the perturbed equation would contain
		\[
		\frac{f(\bar{u})-\bar{v}}{\tau}.
		\]
		Since \(\bar{u}\to \bar{u}_*\) and \(\bar{v}\to \bar{v}_*\) as \(t\to\infty\), this residual does not decay exponentially in time. Consequently, in the energy estimates, this term can only be bounded by a constant of order \(1/\tau\). When applying Gr\"onwall's inequality, such a constant source term would yield an estimate of the form
		\[
		E(t)\le e^{-Ct}E(0)+\frac{C}{\tau}\int_0^t e^{-C(t-s)}\,ds,
		\]
		where the integral produces a term of order \(1/\tau\) that does not decay in time. Thus, one cannot obtain uniform-in-time control of the perturbation, and the convergence can only be established locally in time $t$.
		
		In sharp contrast, by choosing the shifted profiles \((\hat{u},\hat{v})=(\bar{u}-\bar{u}_*,\bar{v}-\bar{v}_*)\), Corollary \ref{corhat} ensures that both \(\hat{u}\) and \(\hat{v}\) decay exponentially to zero. Together with the condition \(f(0)=0\), this implies that \(f(\hat{u})\) also decays exponentially. Therefore, the residual term
		\[
		\frac{f(\hat{u})-\hat{v}}{\tau}
		\]
		is exponentially small in time, which allows us to close the energy estimates uniformly for all \(t\ge 0\) and obtain the global convergence result stated in Theorem \ref{thm3}.
	\end{Rem}
	
	\begin{Rem}[The constant-sensitivity case] If the sensitivity is constant, $g\equiv c^2$ with $c>0$, then $g'\equiv0$ and $h=g+g's\equiv c^2$, so that all the terms involving
		$g'$, $h'$ and $h''$ in the energy estimates of Section \ref{sec3} (and of the Appendix) vanish identically, and no smallness of $N(t)$ is required in the proofs. Consequently, the derivative smallness conditions on the initial data can be removed: the conclusions of Theorems \ref{thm1} and \ref{thm2} hold for arbitrary initial data in $H^3(0,1)\times H^3(0,1)$ satisfying the compatibility conditions \eqref{comp-1}--\eqref{comp-2} (with the relaxation terms dropped for Theorem~\ref{thm2}), and Theorem~\ref{thm3} holds without the smallness condition $\varepsilon_3$, provided $\tau\ge\tau^*$ and the well-preparedness condition in \eqref{goal2} is satisfied. In this case, the constants $C$ and $\tau^*$ may depend on the amplitude of the initial data.
	\end{Rem}

	\section{Global Existence and Long-time Behavior for the Original System}\label{sec3}
	\subsection{Reformulation of the Original Problem}\label{sec3.1}
	Differentiating the first equation of \eqref{a} with respect to $t$ and the second equation of \eqref{a} with respect to $x$, respectively, we reduce the system
	\eqref{a}--\eqref{b} to
	\begin{align}\label{sy1}
		u_{tt}^{\tau}+\frac{1}{\tau}u_{t}^{\tau}-(g(u^{\tau}_x) u^{\tau}_x)_x=-\frac{f(u^{\tau})_x}{\tau}, \quad x\in (0,1),\ t>0,
	\end{align}
	subject to the initial-boundary conditions
	\begin{align}\label{sy2}
		\begin{cases}
			u^{\tau}(x,0) =u_{0}^\tau(x), & x\in (0,1),\\[3pt]
			u_{t}^{\tau}(x,0) =u_{1}^\tau(x):=-v_{0x}^\tau, & x\in (0,1),\\[3pt]
			u_{x}^{\tau}(0,t) =0,\quad u_{x}^{\tau}(1,t) =-ku_{t}^{\tau}(1,t),\ k>0, & t>0.
		\end{cases}
	\end{align}
	Furthermore, $v^\tau$ is determined by
	\begin{align*}
		v_t^\tau+\frac{1}{\tau}v^\tau=-g(u^{\tau}_x)u^{\tau}_x+\frac{f(u^{\tau})}{\tau}, \quad x\in (0,1),\ t>0,
	\end{align*}
	with the initial condition
	\begin{align*}
		v^\tau(x,0)=v_0^\tau(x), \quad x\in(0,1).
	\end{align*}
	
	Now, we introduce the solution space for the problem \eqref{sy1} with \eqref{sy2} as follows:
	$$
	X(0,T) :=\left\{u(x,t):u\in C^0\left([0,T);H^3(0,1)\right) ,\left. u_{t}\in C^0\left([0,T);H^2(0,1) \right)\right\} \right.
	$$
	with $0<T\leq+\infty$.

	To investigate the global existence of solutions to the problem \eqref{a}--\eqref{b}, we begin by establishing the global existence result for the reformulated nonlinear wave equations \eqref{sy1} subject to \eqref{sy2}, which serves as a crucial stepping stone toward the original system.

	\begin{The}\label{thm4}
		Under the assumptions of Theorem \ref{thm1}, there exists a positive constant $\epsilon_1$ such that if $\|u_{0x}^\tau\|_2+\|u_1^\tau\|_2\leq \epsilon_1$, then problem \eqref{sy1}--\eqref{sy2} admits a unique global solution 
		$u^{\tau}\in X(0,+\infty)$.
		Moreover, there exists a constant $C_1>0$, independent of $\tau$ and $t$, such that for all $t\ge 0$,
		\begin{align}
			\|u_{x}^{\tau}\|_{2}+\|u_{t}^{\tau}\|_{2} &\leq C_1 e^{-\alpha t}
			\bigl(\|u_{0x}^{\tau}\|_{2}+\|u_{1}^{\tau}\|_{2}\bigr), \label{est}
		\end{align}
		for any $\alpha\in \left(0,\frac{\lambda}{8\left(1+\frac{\lambda}{\sqrt{\delta}}\right)}\right)$, and
		\begin{align}
			\|u^{\tau}\| \leq 2\|u_{0}^{\tau}\|. \label{eeest}
		\end{align}
		Furthermore, it holds that
		\begin{align}
			\|(u^{\tau}-u^\tau_*)(t)\| &\leq C_1 e^{-\alpha t}
			\bigl(\|u_{0x}^{\tau}\|_{2}+\|u_{1}^{\tau}\|_{2}\bigr), \label{eest}
		\end{align}
		for all $t\ge 0$, where $u^\tau_*$ is given by \eqref{new1}.
	\end{The}

	To prove Theorem \ref{thm4}, we employ a standard continuation argument based on the interplay between local existence and a priori estimates. To this end, we define
	\[
	N(t):=\sup_{0\leq s \leq t}\left\{e^{\alpha s }\left\|u_x^\tau(s)\right\|_2+e^{\alpha s}\left\|u_s^\tau(s)\right\|_2\right\}.
	\]
	Then, by the Sobolev embedding theorem, we have
	\begin{align}\label{eng}
		\sup_{s\in [0,t]}\left\{\left\|u_{x}^{\tau}(s)\right\|_{L^{\infty}}+\left\|u_{xx}^{\tau}(s)\right\|_{L^{\infty}}+\left\|u_{s}^{\tau}(s)\right\|_{L^{\infty}}+\left\|u_{sx}^{\tau}(s)\right\|_{L^{\infty}} \right\} \le CN(t).
	\end{align}
	Moreover, for $u^\tau$ itself, we obtain
	\begin{align*}\label{uinfty}
		\|u^\tau\|\le \|u_0^\tau\|+\int_0^t \|u_s^\tau\| ds\le \|u_0^\tau\|+\frac{N(t)}{\alpha}.
	\end{align*}
	Thus, by the Sobolev embedding theorem again
	\begin{align}
		\sup_{s\in[0,t]}\|u(s)\|_{L^\infty}\le \|u_0^\tau\|+CN(t).
	\end{align}
	
	We shall make use of the following two propositions.
	
	\begin{Pro}[Local existence]\label{local}
		For any $\epsilon_0>0$, there exists a positive constant $T_0$ depending on $\epsilon_0$ and $u_0^\tau$ such that if $u_0\in H^3(0,1)$, $u_1\in H^2(0,1)$ with $N(0)\le\epsilon_0$, then problem \eqref{sy1}--\eqref{sy2} admits a unique solution $u^{\tau}\in X(0,T_0)$ satisfying $N(t)\le 2N(0)$ for any $0\leq t\leq T_0$.
	\end{Pro}
	
	\begin{Pro}[A priori estimate]\label{priori}
		Let $u^{\tau}\in X(0,T)$ be a solution for some $T>0$. Then there exists a positive constant $\epsilon$ independent of $T$ such that if
		\begin{align}\label{pri}
			N(t)\le\epsilon, \quad t\in[0,T],
		\end{align}
		then $u^{\tau}$ satisfies \eqref{est}, \eqref{eest} and
		\begin{align}\label{3.11}
			\|u^{\tau}\|\leq \|u_0^\tau\|+C_2\left( \left\| u_{0x}^{\tau} \right\|+\| u_{1}^{\tau}\| \right),
		\end{align}
		for any $0\leq t\leq T$, where $C_2>0$ is a constant independent of $T$ and $\tau$.
	\end{Pro}
	
	The proof of local existence is quite standard and can be found in \cite{Nishida-1978}, so we omit it here. The {\it a priori} estimates, which constitute the core of our analysis, will be established in the following subsection.

	\subsection{The {\it a priori} estimates}\label{sec3.2}
	To prove Proposition \ref{priori}, we first assume that \eqref{pri} holds for some positive constant $\epsilon$ to be determined later. Our goal is to derive the a priori estimates \eqref{est}, \eqref{eest} and \eqref{3.11} under this assumption, and then verify that the solution indeed satisfies \eqref{pri} with a possibly smaller $\epsilon$, thereby closing the argument. In other words, we may assume $N(t)$ to be sufficiently small throughout this subsection.
	
	To begin with, we introduce several constants that depend on neither $t$ nor $\tau$, which will be used frequently throughout this subsection. Define
	\begin{align}\label{lambda}
		\lambda:=\min\left\{\frac{2k\delta}{1+k^2\delta},\frac{\sqrt{\delta}}{2}\right\}
	\end{align}
	and
	\begin{align}\label{be}
		\beta:=\frac{\lambda^2}{4\left(1+\frac{\lambda}{\sqrt{\delta}}\right)},
	\end{align}
	where $k>0$ and $\delta>0$ are the constants given in \eqref{bb} and the hypothesis \eqref{h1}, respectively.
	
	We now establish the first-order energy estimates, which form the foundation for the subsequent higher-order analysis.
	\begin{Lem}\label{le3.1}
		Under the assumptions of Theorem \ref{thm4}, there exist positive constants $C$ and $\tau^*$, independent of both $\tau$ and $t$, such that the following estimates hold whenever $\tau\geq\tau^*$ and $N(t)$ is sufficiently small:
		\begin{align}\label{es1}
			\|u_{x}^{\tau}\|+\|u_{t}^{\tau}\|\leq Ce^{-\alpha_1 t}\left(\|u_{0x}^{\tau}\|+\|u_{1}^{\tau}\|\right),
		\end{align}
		and
		\begin{align}\label{es2}
			\|u^{\tau}\|\leq \|u_0^\tau\|+C\left( \left\| u_{0x}^{\tau} \right\|+\| u_{1}^{\tau}\| \right),
		\end{align}
		where $\alpha_1:=\frac{\lambda}{8\left(1+\frac{\lambda}{\sqrt{\delta}}\right)}$ is a positive constant independent of $\tau$ and $t$, and $\lambda$ is defined in \eqref{lambda}.
	\end{Lem}
	\begin{proof}
		Multiplying \eqref{sy1} by $u^{\tau}_t$ and integrating the resulting equation with respect to $x$ over $(0,1)$, we obtain
		\begin{align*}
			\frac{1}{2}\frac{d}{dt}\int_0^1{(u_{t}^{\tau}) ^2dx}+\frac{1}{\tau}\int_0^1{(u_{t}^{\tau}) ^2dx}-\int_0^1{(g(u_{x}^{\tau} ) u_{x}^{\tau} ) _x}u_{t}^{\tau}dx=-\frac{1}{\tau}\int_0^1{f'(u^{\tau})}u_{x}^{\tau}u_{t}^{\tau}dx.
		\end{align*}
		Then, integrating by parts and making use of \eqref{sy2}, we have 
		\begin{align*}
			\begin{aligned}
				&-\int_0^1{(g(u_{x}^{\tau} ) u_{x}^{\tau} ) _x}u_{t}^{\tau}dx\\
				&=kg(u_{x}^{\tau}(1,t) ) ( u_{t}^{\tau}(1,t) ) ^2+\int_0^1{g(u_{x}^{\tau} ) u_{x}^{\tau}}u_{tx}^{\tau}dx
				\\
				&=kg(u_{x}^{\tau}(1,t) ) ( u_{t}^{\tau}(1,t) ) ^2+\frac{1}{2}\frac{d}{dt}\int_0^1{g(u_{x}^{\tau} )}(u_{x}^{\tau} ) ^2dx-\frac{1}{2}\int_0^1{g'(u_{x}^{\tau} ) u_{xt}^{\tau}}(u_{x}^{\tau} ) ^2.
			\end{aligned}
		\end{align*}
		Consequently, substituting this back yields:
		\begin{align}\label{1q1}
			&\frac{d}{dt}\left( \frac{1}{2}\int_0^1{(u_{t}^{\tau}) ^2dx}+\frac{1}{2}\int_0^1{g(u_{x}^{\tau} )}(u_{x}^{\tau} ) ^2dx \right) +\frac{1}{\tau}\int_0^1{(u_{t}^{\tau}) ^2dx}+kg(u_{x}^{\tau}(1,t) ) ( u_{t}^{\tau}(1,t) ) ^2
			\nonumber\\
			&=\frac{1}{2}\int_0^1{g'(u_{x}^{\tau} ) u_{xt}^{\tau}}(u_{x}^{\tau} ) ^2-\frac{1}{\tau}\int_0^1{f' (u^{\tau})u_{x}^{\tau}u_{t}^{\tau}dx.}
		\end{align}
		Now let us consider the modified energy $\int_0^1 xu_x^{\tau}u_t^{\tau} d x$. Differentiating it and integrating by parts, applying \eqref{sy1} again, we obtain
		\begin{align}\label{xxx}
			\frac{d}{dt}\int_0^1 xu_x^{\tau}u_t^{\tau} d x=&\int_0^1{xu_{xt}^{\tau}u_{t}^{\tau}dx}+\int_0^1{xu_{x}^{\tau}u_{tt}^{\tau}dx}\nonumber\\
			=&\frac{1}{2}\int_0^1{x\left( (u_{t}^{\tau}) ^2 \right) _xdx}+\int_0^1{xu_{x}^{\tau}\left( -\frac{1}{\tau}u_{t}^{\tau}+\left( g(u_{x}^{\tau})u_{x}^{\tau} \right) _x-\frac{f'(u^{\tau} ) u_{x}^{\tau}}{\tau} \right) dx}\nonumber\\
			=&\frac{1}{2}( u_{t}^{\tau}(1,t) ) ^2-\frac{1}{2}\int_0^1{(u_{t}^{\tau}) ^2dx}-\frac{1}{\tau}\int_0^1{xu_{x}^{\tau}u_{t}^{\tau}dx}\nonumber\\
			& +\int_0^1{xu_{x}^{\tau}\left( g(u_{x}^{\tau})u_{x}^{\tau} \right) _xdx}-\frac{1}{\tau}\int_0^1{xf'(u^{\tau} ) (u_{x}^{\tau})^2dx},
		\end{align}
		where 
		\begin{align*}
			&\int_0^1{xu_{x}^{\tau}\left( g(u_{x}^{\tau})u_{x}^{\tau} \right) _xdx}\\
			&=k^2g(u_{x}^{\tau}(1,t) ) ( u_{t}^{\tau}(1,t) ) ^2-\int_0^1{g(u_{x}^{\tau})(u_{x}^{\tau} ) ^2dx}-\frac{1}{2}\int_0^1{xg(u_{x}^{\tau})\left( (u_{x}^{\tau} ) ^2 \right) _xdx}\\
			&=k^2g(u_{x}^{\tau}(1,t) ) ( u_{t}^{\tau}(1,t) ) ^2-\int_0^1{g(u_{x}^{\tau})(u_{x}^{\tau} ) ^2dx}
			-\frac{k^2}{2}g(u_{x}^{\tau}(1,t) ) ( u_{t}^{\tau}(1,t) ) ^2\\
			&\ \ \ +\frac{1}{2}\int_0^1{\left( xg(u_{x}^{\tau}) \right) _x(u_{x}^{\tau} ) ^2dx}\\
			&=\frac{k^2}{2}g(u_{x}^{\tau}(1,t) ) ( u_{t}^{\tau}(1,t) ) ^2-\frac{1}{2}\int_0^1{g(u_{x}^{\tau})(u_{x}^{\tau} ) ^2dx}+\frac{1}{2}\int_0^1{xg' (u_{x}^{\tau})u_{xx}^{\tau}(u_{x}^{\tau} ) ^2dx}.
		\end{align*}
		Substituting the above identity into \eqref{xxx}, we obtain
		\begin{align}\label{1q2}
			&\frac{d}{dt}\int_0^1{xu_{x}^{\tau}u_{t}^{\tau}dx}+\frac{1}{\tau}\int_0^1{xu_{x}^{\tau}u_{t}^{\tau}dx}+\frac{1}{2}\int_0^1{(u_{t}^{\tau}) ^2dx}+\frac{1}{2}\int_0^1{g(u_{x}^{\tau})(u_{x}^{\tau} ) ^2dx}\nonumber
			\\
			&=\left[ \frac{1}{2}+\frac{k^2}{2}g(u_{x}^{\tau}(1,t) ) \right] ( u_{t}^{\tau}(1,t) ) ^2+\frac{1}{2}\int_0^1{xg' (u_{x}^{\tau})u_{xx}^{\tau}(u_{x}^{\tau} ) ^2dx}-\frac{1}{\tau}\int_0^1{xf'(u^{\tau} ) (u_{x}^{\tau})^2dx.}
		\end{align}
		Multiplying \eqref{1q2} by $\lambda$ defined in \eqref{lambda} and adding the result to \eqref{1q1}, we obtain
		\begin{align}\label{p}
			&\frac{d}{dt}\left( \frac{1}{2}\int_0^1{(u_{t}^{\tau}) ^2dx}+\frac{1}{2}\int_0^1{g(u_{x}^{\tau} ) (u_{x}^{\tau} ) ^2dx}+\lambda \int_0^1{xu_{x}^{\tau}u_{t}^{\tau}dx} \right) +\frac{\lambda}{\tau}\int_0^1{xu_{x}^{\tau}u_{t}^{\tau}dx}\nonumber\\
			&\quad +\left( \frac{1}{\tau}+\frac{\lambda}{2} \right) \int_0^1{(u_{t}^{\tau}) ^2dx}+\frac{\lambda}{2}\int_0^1{g(u_{x}^{\tau})(u_{x}^{\tau} ) ^2dx}+B_1(t)\nonumber\\
			&=\underset{I_1}{\underbrace{\frac{1}{2}\int_0^1{g'(u_{x}^{\tau})u_{xt}^{\tau}(u_{x}^{\tau})^2}+\frac{\lambda}{2}\int_0^1{xg'(u_{x}^{\tau})u_{xx}^{\tau}(u_{x}^{\tau})^2dx}}}\nonumber\\
			&\quad\underset{I_2}{\underbrace{-\frac{1}{\tau}\int_0^1{f'(u^{\tau})u_{x}^{\tau}u_{t}^{\tau}dx}}}\underset{I_3}{\underbrace{-\frac{\lambda}{\tau}\int_0^1{xf'(u^{\tau})(u_{x}^{\tau})^2dx}}},
		\end{align}
		where
		\begin{align}\label{B1}
			B_1(t):=\left[ kg(u_{x}^{\tau}(1,t) ) -\left( \frac{\lambda}{2}+\frac{\lambda k^2}{2}g(u_{x}^{\tau}(1,t) ) \right) \right] ( u_{t}^{\tau}(1,t) ) ^2\geq 0,
		\end{align}
		since by \eqref{h1} and \eqref{lambda}, we have
		$$\lambda\le\frac{2k\delta}{1+k^2\delta}\le \frac{2kg(u_{x}^{\tau}(1,t))}{1+k^2g(u_{x}^{\tau}(1,t))}.$$ 
		For the term $I_1$ on the right-hand side of \eqref{p}, in view of \eqref{eng} and \eqref{h1}, we have 
		\begin{align}\label{I1}
			I_1\le \frac{CN(t)(1+\lambda)}{\delta}\int_0^1{g(u_{x}^{\tau} ) (u_{x}^{\tau} ) ^2dx}\leq \frac{\lambda}{8}\int_0^1{g(u_{x}^{\tau} ) (u_{x}^{\tau} ) ^2dx},
		\end{align}
		provided $N(t)\le \frac{C\lambda\delta}{1+\lambda}$ for some constant $C>0$ independent of $\tau$ and $t$.
		For the terms $I_2$ and $I_3$,   using \eqref{uinfty} and  Young's inequality, we find that
		\begin{align}\label{I21}
			I_2&\le \frac{C}{\tau \sqrt{\delta}}\int_0^1{\sqrt{g(u_{x}^{\tau})}\left| u_{x}^{\tau} \right|\left| u_{t}^{\tau} \right|dx}\nonumber\\
			&\le \frac{\lambda}{8}\int_0^1{g(u_{x}^{\tau})(u_{x}^{\tau})^2dx}+\frac{C}{\tau ^2\delta \lambda}\int_0^1{(u_{t}^{\tau})^2dx}\nonumber\\
			&\le \frac{\lambda}{8}\int_0^1{g(u_{x}^{\tau})(u_{x}^{\tau})^2dx}+\frac{1}{\tau}\int_0^1{(u_{t}^{\tau})^2dx},
		\end{align}
		and 
		\begin{align}\label{I22}
			I_3\le \frac{C\lambda}{\tau \delta} \int_0^1{g(u_{x}^{\tau} ) (u_{x}^{\tau} ) ^2dx}\le \frac{\lambda}{8}\int_0^1g(u_{x}^{\tau} ) (u_{x}^{\tau} ) ^2dx,
		\end{align}
		provided $\tau\geq C\max\{\frac{1}{\lambda\delta},\frac{1}{\delta}\}$ for some constant $C>0$ independent of $\tau$ and $t$.
		
		Now, let us define 
		\[\mathcal{F}_1(t):=\frac{1}{2}\int_0^1{(u_{t}^{\tau}) ^2dx}+\frac{1}{2}\int_0^1{g(u_{x}^{\tau} ) (u_{x}^{\tau} ) ^2dx}+\lambda \int_0^1{xu_{x}^{\tau}u_{t}^{\tau}dx}.
		\]
		Substituting \eqref{B1}--\eqref{I22} into \eqref{p}, we  obtain
		\begin{align*}
			\frac{d}{dt}\mathcal{F}_1(t)+\frac{\lambda}{8}\int_0^1{(u_{t}^{\tau}) ^2dx}+\frac{\lambda}{8}\int_0^1{g(u_{x}^{\tau} ) (u_{x}^{\tau} ) ^2dx}+\frac{\lambda}{\tau}\int_0^1{xu_{x}^{\tau}u_{t}^{\tau}dx}\le 0,
		\end{align*}
		provided $N(t)\le \frac{C\lambda\delta}{1+\lambda}$ and $\tau\geq C\max\{\frac{1}{\lambda\delta},\frac{1}{\delta}\}$.
		Hence, applying Young's inequality and recalling the defination of $\beta$ in \eqref{be}, we obtain
		\begin{align}\label{ppp}
			&\frac{d}{dt}\mathcal{F}_1(t)+\frac{\lambda}{8}\int_0^1{(u_{t}^{\tau}) ^2dx}+\frac{\lambda}{8}\int_0^1{g(u_{x}^{\tau} ) (u_{x}^{\tau} ) ^2dx}+\beta \int_0^1{xu_{x}^{\tau}u_{t}^{\tau}dx}\nonumber
			\\
			&\le\left(\beta -\frac{\lambda}{\tau} \right) \int_0^1{xu_{x}^{\tau}u_{t}^{\tau}dx}
			\nonumber\\
			&\le \beta \int_0^1{\left| u_{x}^{\tau} \right|\left| u_{t}^{\tau} \right|dx}
			\nonumber\\
			&\le \frac{\beta}{2\sqrt{\delta}}\int_0^1{(u_{t}^{\tau}) ^2dx}+\frac{\beta}{2\sqrt{\delta}}\int_0^1{g(u_{x}^{\tau} ) (u_{x}^{\tau} ) ^2dx},
		\end{align}
		provided $N(t)\le \frac{C\lambda\delta}{1+\lambda}$ and $\tau\geq \max\{\frac{C}{\lambda\delta},\frac{C}{\delta},\frac{\beta}{\lambda}\}$.
		Thus, it follows from \eqref{ppp} that 
		\begin{align*}
			&\frac{d}{dt}\mathcal{F}_1(t)+\left( \frac{\lambda}{4}-\frac{\beta}{\sqrt{\delta}} \right) \frac{1}{2}\int_0^1{(u_{t}^{\tau}) ^2dx}\\&+\left( \frac{\lambda}{4}-\frac{\beta}{\sqrt{\delta}} \right) \frac{1}{2}\int_0^1{g(u_{x}^{\tau} ) (u_{x}^{\tau} ) ^2dx}
			+\frac{\beta}{\lambda}\left( \lambda \int_0^1{xu_{x}^{\tau}u_{t}^{\tau}dx} \right) \le 0.
		\end{align*}
		Since $\frac{\lambda}{4}-\frac{\beta}{\sqrt{\delta}}= \frac{\beta}{\lambda}$ by \eqref{be}, it follows that 
		\begin{align*}
			\frac{d}{dt}\mathcal{F}_1(t)+\frac{\beta}{\lambda} \mathcal{F}_1(t)\le 0.
		\end{align*}
		Thus, Gr\"onwall's inequality yields $\mathcal{F}_1(t)\leq Ce^{-\frac{\beta}{\lambda}t}\mathcal{F}_1(0)$, which implies
		\begin{align}\label{111}
			&\frac{1}{2}\int_0^1{(u_{t}^{\tau}) ^2dx}+\frac{1}{2}\int_0^1{g(u_{x}^{\tau} ) (u_{x}^{\tau} ) ^2dx}+\lambda \int_0^1{xu_{x}^{\tau}u_{t}^{\tau}dx}
			\nonumber\\
			&\le Ce^{-\frac{\beta}{\lambda} t}\left( \frac{1}{2}\int_0^1{(u_{1}^{\tau}) ^2dx}+\frac{1}{2}\int_0^1{g(u_{0x}^{\tau}) (u_{0x}^{\tau})^2}+\lambda \int_0^1{xu_{0x}^{\tau}u_{1}^{\tau}dx} \right) 
			\nonumber\\
			&\le Ce^{-\frac{\beta}{\lambda} t}\left( \int_0^1{(u_{1}^{\tau}) ^2dx}+\int_0^1{(u_{0x}^{\tau})^2dx} \right),
		\end{align}
		provided $N(t)\le \frac{C\lambda\delta}{1+\lambda}$ and $\tau\geq \max\{\frac{C}{\lambda\delta},\frac{C}{\delta},\frac{\beta}{\lambda}\}$, where we have used Young's inequality. By using Young's inequality again and taking advantage of \eqref{h1} and \eqref{lambda}, we see that 
		\begin{align*}
			\begin{aligned}
				-\lambda \int_0^1{xu_{x}^{\tau}u_{t}^{\tau}dx}&\le \frac{\lambda}{\sqrt{\delta}}\int_0^1{\sqrt{g(u_{x}^{\tau})}\left| u_{x}^{\tau} \right|\left| u_{t}^{\tau} \right|dx}\\
				&\le \frac{\lambda}{2\sqrt{\delta}}\int_0^1{(u_{t}^{\tau})^2dx}+\frac{\lambda}{2\sqrt{\delta}}\int_0^1{g(u_{x}^{\tau})(u_{x}^{\tau})^2dx}\\
				&\le \frac{1}{4}\int_0^1{(u_{t}^{\tau})^2dx}+\frac{1}{4}\int_0^1{g(u_{x}^{\tau})(u_{x}^{\tau})^2dx}.
			\end{aligned}
		\end{align*}
		Substituting the above inequality into \eqref{111}, we have
		\begin{align}\label{jie}
			\int_0^1{(u_{t}^{\tau}) ^2dx}+\int_0^1{g(u_{x}^{\tau} ) (u_{x}^{\tau} ) ^2dx}\le Ce^{-\frac{\beta}{\lambda} t}\left( \int_0^1(u_{1}^{\tau})^2dx+\int_0^1{(u_{0x}^{\tau})^2dx} \right)
		\end{align}
		provided $N(t)\le \frac{C\lambda\delta}{1+\lambda}$ and $\tau\geq \max\{\frac{C}{\lambda\delta},\frac{C}{\delta},\frac{\beta}{\lambda}\}$.
		Now set $\tau^*:= \max\{\frac{C}{\lambda\delta},\frac{C}{\delta},\frac{\beta}{\lambda}\}$ and $\alpha_1:=\frac{\beta}{2\lambda}=\frac{\lambda}{8(1+\frac{\lambda}{\sqrt{\delta}})}  $. Then \eqref{es1} follows from \eqref{jie} and \eqref{h1}.  
		
		On the other hand, to estimate $u^\tau$ itself, noting that $u^{\tau}-u_{0}^{\tau}=\int_0^t u_{s}^{\tau}(x,s)ds$, and applying \eqref{es1}, we obtain 
		\begin{align*}
			\begin{aligned}
				\left\| u^{\tau} \right\|&\le \left\| u_{0}^{\tau} \right\|+\int_0^t{\left\| u_{s}^{\tau}\left( \cdot,s \right) \right\|}ds
				\\
				&\le \left\| u_{0}^{\tau} \right\|+C\int_0^t{e^{-\alpha s}\left( \left\| u_{0x}^{\tau} \right\|+\| u_{1}^{\tau}\| \right) ds}
				\\
				&\le \left\| u_{0}^{\tau} \right\|+\frac{C}{\alpha}\left( 1-e^{-\alpha t} \right) \left( \left\| u_{0x}^{\tau} \right\|+\| u_{1}^{\tau}\| \right) 
				\\
				&\le \|u_0^\tau\|+C\left( \left\| u_{0x}^{\tau} \right\|+\| u_{1}^{\tau}\| \right),
			\end{aligned}
		\end{align*}
		which implies \eqref{es2}. This completes the proof of Lemma \ref{le3.1}.
	\end{proof}
	
	Having established the first-order energy estimates, we now proceed to derive the second-order estimates.
	\begin{Lem}\label{le3.2}
		Under the assumptions of Theorem \ref{thm4}, there exist positive constants $C$ and $\tau^*$, independent of both $\tau$  and $t$, such that the following estimate holds whenever $\tau\geq\tau^*$ and $N(t)$  is sufficiently small: 
		\begin{align}\label{es3}
			\|u_{xt}^{\tau}\|+\|u_{tt}^{\tau}\|\le Ce^{-\alpha_2 t}\left( \left\| u_{0x}^\tau \right\|_1+\| u_{1}^\tau\|_1 \right),
		\end{align}
		for any $\alpha_2\in (0, \alpha_1)  $, where $\alpha_1=\frac{\lambda }{8\left( 1+\frac{\lambda}{\sqrt{\delta}} \right)}$ is defined in Lemma \ref{le3.1}.
	\end{Lem}
	
	\begin{proof}
		Differentiating \eqref{sy1} with respect to $t$ and setting $V^{\tau}:=u^\tau_t$, we obtain
		\begin{align}\label{sy3}
			V_{tt}^{\tau}+\frac{1}{\tau}V_{t}^{\tau}-(h(u^\tau_x) V_{x}^{\tau}) _x=-\frac{1}{\tau}(f'(u^{\tau} ) V^{\tau}  ) _x,
		\end{align}
		where $h(s):= g'(s)s+g(s) $.
		The corresponding boundary conditions are
		\begin{align}\label{e}
			V_{x}^{\tau}\left( 0,t \right) =0,~ V_{x}^{\tau}(1,t) =-kV_{t}^{\tau}(1,t),
		\end{align}
		and the initial data are
		\begin{align}\label{Vini}
			\left\{
			\begin{aligned}
				V^\tau(x,0)=V_0^\tau(x):=& u_1^\tau, \\
				V_t^\tau(x,0)=V_1^\tau(x):=& -\frac{1}{\tau}(u_1^\tau+f'(u_0^\tau)u_{0x})+h(u_{0x}^\tau)u_{0xx}^\tau
			\end{aligned}
			\right.
		\end{align}
		By the assumptions \eqref{h1} and \eqref{h2}, we have
		\begin{align}\label{hs}
			h(s)\ge \delta>0, \quad \text{for all } s\in \mathbb{R}.
		\end{align}
		
		Multiplying \eqref{sy3} by $V^{\tau}_t$, and integrating the result with respect to $x$ over $(0,1)$, we derive
		\begin{align}\label{m1}
			\frac{1}{2}\frac{d}{dt}\int_0^1{(V_{t}^{\tau}  ) ^2dx}+\frac{1}{\tau}\int_0^1{(V_{t}^{\tau}  ) ^2dx}-\int_0^1{(h(u_{x}^{\tau} ) V_{x}^{\tau} ) _xV_{t}^{\tau}dx}=-\frac{1}{\tau}\int_0^1{( f'(u^{\tau} ) V^{\tau}) _xV_{t}^{\tau}dx.}
		\end{align}
		For the last term on the left-hand side of \eqref{m1}, integrating by parts and using \eqref{e}, we obtain 
		\begin{align*}
			&-\int_0^1{(h(u_{x}^{\tau} ) V_{x}^{\tau} ) _xV_{t}^{\tau}dx}\\
			&=kh(u_{x}^{\tau}(1,t) ) (V_{t}^{\tau}(1,t)) ^2+\int_0^1{h(u_{x}^{\tau} ) V_{x}^{\tau}V_{xt}^{\tau}dx}\\
			&=kh(u_{x}^{\tau}(1,t) ) (V_{t}^{\tau}(1,t)) ^2+\frac{1}{2}\frac{d}{dt}\int_0^1{h(u_{x}^{\tau} ) (V_{x}^{\tau} ) ^2dx}-\frac{1}{2}\int_0^1{h'(u_{x}^{\tau} ) V_{x}^{\tau}(V_{x}^{\tau} ) ^2.}
		\end{align*}
		Consequently,
		\begin{align}\label{2q1}
			&\frac{d}{dt}\left( \frac{1}{2}\int_0^1{(V_{t}^{\tau}  ) ^2dx}+\frac{1}{2}\int_0^1{h(u_{x}^{\tau} ) (V_{x}^{\tau} ) ^2dx} \right) +\frac{1}{\tau}\int_0^1{(V_{t}^{\tau}  ) ^2dx}+kh(u_{x}^{\tau}(1,t) ) (V_{t}^{\tau}(1,t)) ^2\nonumber\\
			&=\frac{1}{2}\int_0^1{h'(u_{x}^{\tau} ) V_{x}^{\tau}(V_{x}^{\tau} ) ^2}-\frac{1}{\tau}\int_0^1{f''(u^{\tau} ) u_{x}^{\tau}V^{\tau}V_{t}^{\tau}dx}-\frac{1}{\tau}\int_0^1{f'(u^{\tau} ) V_{x}^{\tau}}V_{t}^{\tau}dx.
		\end{align}
		Proceeding as in Lemma \ref{le3.1}, we have
		\begin{align}\label{m2}
			\frac{d}{dt}\int_0^1{xV_{x}^{\tau}V_{t}^{\tau}dx}
			=&\int_0^1{xV_{xt}^{\tau}V_{t}^{\tau}dx}+\int_0^1{xV_{x}^{\tau}V_{tt}^{\tau}dx}\nonumber\\
			=&\frac{1}{2}\int_0^1{x( (V_{t}^{\tau}  ) ^2) _xdx}+\int_0^1{xV_{x}^{\tau}\left( -\frac{1}{\tau}V_{t}^{\tau}+(h(u_{x}^{\tau})V_{x}^{\tau}) _x-\frac{(f'(u^{\tau} ) V^{\tau}  ) _x}{\tau} \right) dx}\nonumber\\
			=&\frac{1}{2}(V_{t}^{\tau}(1,t)) ^2-\frac{1}{2}\int_0^1{(V_{t}^{\tau}  ) ^2dx}-\frac{1}{\tau}\int_0^1{xV_{x}^{\tau}V_{t}^{\tau}dx}\nonumber\\
			& +\int_0^1{xV_{x}^{\tau}(h(u_{x}^{\tau})V_{x}^{\tau}) _xdx}-\frac{1}{\tau}\int_0^1{x(f'(u^{\tau} ) V^{\tau}  ) _xV_{x}^{\tau}dx.}
		\end{align}
		Next, we simplify the last two terms on the right-hand side of \eqref{m2}. One has
		\begin{align}\label{qq}
			\int_0^1{xV_{x}^{\tau}(h(u_{x}^{\tau})V_{x}^{\tau}) _xdx}=&k^2h(u_{x}^{\tau}(1,t) ) (V_{t}^{\tau}(1,t)) ^2-\int_0^1{h(u_{x}^{\tau})(V_{x}^{\tau} ) ^2dx}-\frac{1}{2}\int_0^1{xh(u_{x}^{\tau})((V_{x}^{\tau} ) ^2) _xdx}\nonumber\\
			=&k^2h(u_{x}^{\tau}(1,t) ) (V_{t}^{\tau}(1,t)) ^2-\int_0^1{h(u_{x}^{\tau})(V_{x}^{\tau} ) ^2dx}
			-\frac{k^2}{2}h(u_{x}^{\tau}(1,t) ) (V_{t}^{\tau}(1,t)) ^2\nonumber\\
			&+\frac{1}{2}\int_0^1{\left( xh(u_{x}^{\tau}) \right) _x(V_{x}^{\tau} ) ^2dx}\nonumber\\
			=&\frac{k^2}{2}h(u_{x}^{\tau}(1,t) ) (V_{t}^{\tau}(1,t)) ^2-\frac{1}{2}\int_0^1{h(u_{x}^{\tau})(V_{x}^{\tau} ) ^2dx}\nonumber\\
			&+\frac{1}{2}\int_0^1 xh'(u_{x}^{\tau})u_{xx}^{\tau}(V_{x}^{\tau} ) ^2dx,
		\end{align}
		and 
		\begin{align}\label{qqq}
			-\frac{1}{\tau}\int_0^1{x(f'(u^{\tau} ) V^{\tau}  ) _xV_{x}^{\tau}dx}
			=-\frac{1}{\tau}\int_0^1{xf'' (u^{\tau} ) u_{x}^{\tau}V^{\tau}V_{x}^{\tau}dx}-\frac{1}{\tau}\int_0^1{xf'(u^{\tau} ) (V_{x}^{\tau} ) ^2}.
		\end{align}
		Hence, substituting \eqref{qq} and \eqref{qqq} into \eqref{m2} yields
		\begin{align}\label{2q2}
			&\frac{d}{dt}\int_0^1{xV_{x}^{\tau}V_{t}^{\tau}dx}+\frac{1}{\tau}\int_0^1{xV_{x}^{\tau}V_{t}^{\tau}dx}+\frac{1}{2}\int_0^1{(V_{t}^{\tau}  ) ^2dx}+\frac{1}{2}\int_0^1{h(u_{x}^{\tau})(V_{x}^{\tau} ) ^2dx}\nonumber\\
			&=\left[ \frac{1}{2}+\frac{k^2}{2}h(u_{x}^{\tau}(1,t) ) \right] (V_{t}^{\tau}(1,t)) ^2+\frac{1}{2}\int_0^1{xh' (u_{x}^{\tau})u_{xx}^{\tau}(V_{x}^{\tau} ) ^2dx}\nonumber\\
			&\quad-\frac{1}{\tau}\int_0^1{xf'' (u^{\tau} ) u_{x}^{\tau}V^{\tau}V_{x}^{\tau}dx}-\frac{1}{\tau}\int_0^1{xf'(u^{\tau} ) (V_{x}^{\tau} ) ^2}.
		\end{align}
		Thus, multiplying \eqref{2q2} by $\lambda$ defined in \eqref{lambda} and adding the result to \eqref{2q1}, we obtain 
		\begin{align}\label{p1}
			&\frac{d}{dt}\left( \frac{1}{2}\int_0^1{(V_{t}^{\tau})^2dx}+\frac{1}{2}\int_0^1{h(u_{x}^{\tau})(V_{x}^{\tau})^2dx}+\lambda \int_0^1{xV_{x}^{\tau}V_{t}^{\tau}dx} \right) +\frac{\lambda}{\tau}\int_0^1{xV_{x}^{\tau}V_{t}^{\tau}dx}\nonumber\\
			&\quad +\left( \frac{1}{\tau}+\frac{\lambda}{2} \right) \int_0^1{(V_{t}^{\tau})^2dx}+\frac{\lambda}{2}\int_0^1{h(u_{x}^{\tau})(V_{x}^{\tau})^2dx}+B_2(t)\nonumber\\
			&=\underset{J_1}{\underbrace{\frac{1}{2}\int_0^1{h'(u_{x}^{\tau})V_{x}^{\tau}(V_{x}^{\tau})^2}+\frac{\lambda}{2}\int_0^1{xh'(u_{x}^{\tau})u_{xx}^{\tau}(V_{x}^{\tau})^2dx}}}\nonumber\\
			&\quad \underset{J_2}{\underbrace{-\frac{1}{\tau}\int_0^1{f'' (u^{\tau})u_{x}^{\tau}V^{\tau}V_{t}^{\tau}dx}-\frac{1}{\tau}\int_0^1{f' (u^{\tau})V_{x}^{\tau}V_{t}^{\tau}dx}-\frac{\lambda}{\tau}\int_0^1{xf'' (u^{\tau})u_{x}^{\tau}V^{\tau}V_{x}^{\tau}dx}}}\nonumber\\
			&\quad \underset{J_3}{\underbrace{-\frac{\lambda}{\tau}\int_0^1{xf' (u^{\tau})(V_{x}^{\tau})^2}}},
		\end{align}
		where
		\[
		B_2(t) :=\left[ kh(u_{x}^{\tau}(1,t) ) -\left( \frac{\lambda}{2}+\frac{\lambda k^2}{2}h(u_{x}^{\tau}(1,t) ) \right) \right] (V_{t}^{\tau}(1,t)) ^2\ge 0,
		\]
		follows from \eqref{lambda} and \eqref{hs}.
		
		For the term $J_1$ on the right-hand side of \eqref{p1}, in view of \eqref{eng} and \eqref{hs}, one has
		\begin{align}\label{J1}
			J_1\le \frac{CN(t)(1+\lambda)}{\delta}\int_0^1{h(u_{x}^{\tau} ) (V_{x}^{\tau} ) ^2dx}\le \frac{\lambda}{8}\int_0^1{h(u_{x}^{\tau} ) (V_{x}^{\tau} ) ^2dx},
		\end{align}
		provided that $N(t)$ is choosen sufficiently small such that $\frac{CN(t)(1+\lambda)}{\delta}\le \frac{\lambda}{8}$.
		
		For the terms $J_2$ and $J_3$, using Young's inequality and \eqref{hs},  we find that
		\begin{align}\label{J21}
			J_2\le & \frac{C}{\tau}\int_0^1{\left| V^{\tau} \right|\left| V_{t}^{\tau} \right|dx}+\frac{C}{\tau \sqrt{\delta}}\int_0^1{\sqrt{h(u_{x}^{\tau})}\left| V_{x}^{\tau} \right|\left| V_{t}^{\tau} \right|dx}+\frac{C}{\tau \sqrt{\delta}}\int_0^1{\sqrt{h(u_{x}^{\tau})}\left| V^{\tau} \right|\left| V_{x}^{\tau} \right|dx}\nonumber\\
			\le & \frac{\lambda}{8}\int_0^1{h(u_{x}^{\tau})(V_{x}^{\tau})^2dx}+\left( \frac{C}{\tau ^2\delta \lambda}+\frac{1}{2\tau} \right) \int_0^1{(V_{t}^{\tau})^2dx}+\left(\frac{C}{\tau ^2\delta \lambda}+ \frac{C}{2\tau} \right) \int_0^1{(V^{\tau})^2dx}\nonumber\\
			\le & \frac{\lambda}{8}\int_0^1{h(u_{x}^{\tau})(V_{x}^{\tau})^2dx}+\left( \frac{C}{\tau ^2\delta \lambda}+\frac{1}{2\tau} \right) \int_0^1{(V_{t}^{\tau})^2dx}+C \int_0^1{(V^{\tau})^2dx}\nonumber\\
			\le & \frac{\lambda}{8}\int_0^1{h(u_{x}^{\tau})(V_{x}^{\tau})^2dx}+\frac{1}{\tau}\int_0^1{(V_{t}^{\tau})^2dx}+C \int_0^1{(V^{\tau})^2dx},
		\end{align}
		and 
		\begin{align}\label{J22}
			J_3\le \frac{C\lambda}{\tau \delta} \int_0^1{h(u_{x}^{\tau} ) (V_{x}^{\tau} ) ^2dx}\le \frac{\lambda}{8}\int_0^1{h(u_{x}^{\tau} ) (V_{x}^{\tau} ) ^2dx},
		\end{align}
		provided $\tau\geq C\max\{\frac{1}{\lambda\delta},\frac{1}{\delta}\}$.
		
		Now let us define 
		\[
		\mathcal{F}_2(t):=\frac{1}{2}\int_0^1{(V_{t}^{\tau}  ) ^2dx}+\frac{1}{2}\int_0^1{h(u_{x}^{\tau} ) (V_{x}^{\tau} ) ^2dx}+\lambda \int_0^1{xV_{x}^{\tau}V_{t}^{\tau}dx}.
		\]
		Substituting \eqref{J1}--\eqref{J22} into \eqref{p1}, we  obtain
		\begin{align*}
			\frac{d}{dt}\mathcal{F}_2(t)+\frac{\lambda}{8}\int_0^1{(V_{t}^{\tau}  ) ^2dx}+\frac{\lambda}{8}\int_0^1{h(u_{x}^{\tau} ) (V_{x}^{\tau} ) ^2dx}+\frac{\lambda}{\tau}\int_0^1{xV_{x}^{\tau}V_{t}^{\tau}dx}\le C \int_0^1{(V^{\tau})^2dx}.
		\end{align*}
		Hence, with $\beta$ defined in \eqref{be} and assuming $\tau>\frac{\lambda}{\beta}$, applying Young's inequality gives
		\begin{align*}
			\begin{aligned}
				&\frac{d}{dt}\mathcal{F}_2(t)+\frac{\lambda}{8}\int_0^1{(V_{t}^{\tau}  ) ^2dx}+\frac{\lambda}{8}\int_0^1{h(u_{x}^{\tau} ) (V_{x}^{\tau} ) ^2dx}+\beta \int_0^1{xV_{x}^{\tau}V_{t}^{\tau}dx}\\
				&\le\left(\beta -\frac{\lambda}{\tau}\right)\int_0^1{xV_{x}^{\tau}V_{t}^{\tau}dx}+C \int_0^1{(V^{\tau})^2dx}\\
				&\le \beta \int_0^1{\left| V_{x}^{\tau} \right|\left| V_{t}^{\tau} \right|dx}+C \int_0^1{(V^{\tau})^2dx}\\
				&\le \frac{\beta}{2\sqrt{\delta}}\int_0^1{(V_{t}^{\tau}  ) ^2dx}+\frac{\beta}{2\sqrt{\delta}}\int_0^1 h(u_{x}^{\tau} ) (V_{x}^{\tau} ) ^2dx+C \int_0^1{(V^{\tau})^2dx}.
			\end{aligned}
		\end{align*}
		This implies that
		\begin{align*}
			\begin{aligned}
				\frac{d}{dt}\mathcal{F}_2(t)&+\left( \frac{\lambda}{4}-\frac{\beta}{\sqrt{\delta}} \right) \frac{1}{2}\int_0^1{(V_{t}^{\tau}  ) ^2dx}+\left( \frac{\lambda}{4}-\frac{\beta}{\sqrt{\delta}} \right) \frac{1}{2}\int_0^1{h(u_{x}^{\tau} ) (V_{x}^{\tau} ) ^2dx}\\
				&+\frac{\beta}{\lambda} \left( \lambda \int_0^1{xV_{x}^{\tau}V_{t}^{\tau}dx} \right) \le C\int_0^1{(V^{\tau})^2dx}.
			\end{aligned}
		\end{align*}
		Thus, from Lemma \ref{le3.1}, it follows that
		\begin{align*}
			\frac{d}{dt}\mathcal{F}_2(t)+2\alpha_1 \mathcal{F}_2(t)\le C\int_0^1{(V^{\tau})^2dx}\le C\left( \left\| u_{0x}^\tau \right\|^{2}+\| u_{1}^\tau\|^{2} \right)e^{-2\alpha_1 t},
		\end{align*}
		provided $N(t)\le \frac{C\lambda\delta}{1+\lambda}$ and $\tau\geq \max\{\frac{C}{\lambda\delta},\frac{C}{\delta},\frac{\beta}{\lambda}\}$. Hence, taking $\tau\ge \tau^*$, where $\tau^*>0$ is a constant independent of $t$ and $\tau$, an application of  Gr\"onwall's inequality together with Lemma \ref{le3.1} yields
		\begin{align}\label{F2est}
			\mathcal{F}_2(t) \le &\left( \mathcal{F}_2(0)+C\left(\left\| u_{0x}^\tau \right\|_1^{2}+\| u_{1}^\tau\|^{2}\right) t \right) e^{-2\alpha_1 t}\nonumber\\
			\le & \left(\frac{1}{2}\int_0^1{(V_{1}^{\tau}  ) ^2dx}+\frac{1}{2}\int_0^1{h(u_{0x}^{\tau} ) (V_{0x}^{\tau} ) ^2dx}+\lambda \int_0^1{xV_{0x}^{\tau}V_{1}^{\tau}dx}\right) e^{-2\alpha_1 t} \nonumber\\
			& + C\left(\left\| u_{0x}^\tau \right\|^{2}+\| u_{1}^\tau\|^{2}\right)e^{-2\alpha_2 t}\nonumber\\
			\le & C\left( \| V_{0x}^{\tau}\|^2+\|V_1^{\tau}\|^2+\left\| u_{0x}^{\tau} \right\|^{2}+\|u_{1}^{\tau}\|^2 \right) e^{-2\alpha _2t},
		\end{align}
		for any $0<\alpha_2< \alpha_1$.
		
		As in Lemma \ref{le3.1}, combining \eqref{F2est} with \eqref{Vini} and \eqref{hs} shows that \eqref{es3} holds. This completes the proof of Lemma \ref{le3.2}.
	\end{proof}
	
	With the first- and second-order estimates in hand, we now turn to the highest-order energy estimates.
	\begin{Lem}\label{le3.3}
		Under the assumptions of Theorem \ref{thm4}, there exist positive constants $C$ and $\tau^*$, independent of both $\tau$  and $t$, such that the following estimate holds whenever $\tau\geq\tau^*$ and $N(t)$  is sufficiently small: 
		\begin{align}\label{es40}
			\|u_{xtt}^{\tau}\|+\|u_{ttt}^{\tau}\|\le Ce^{-\alpha_2 t}\left( \left\| u_{0x}^\tau \right\|_2+\| u_{1}^\tau\|_2 \right),
		\end{align}
		where $\alpha_2$ is  defined in Lemma \ref{le3.2}.
	\end{Lem}
	\begin{proof}
		Taking the second derivative of \eqref{sy1} with respect to $t$ and setting $W^{\tau}:= V_t^\tau=u^\tau_{tt}$, we obtain
		\begin{align*}
			W_{tt}^{\tau}+\frac{1}{\tau}W_{t}^{\tau}-\left( h(u_{x}^{\tau})W_{x}^{\tau} \right) _x=\left( h'(u_{x}^{\tau})V_{x}^{2} \right) _x-\frac{1}{\tau}\left( f'' (u^{\tau})\left( V^{\tau} \right) ^2+f'(u^{\tau})W^{\tau} \right) _x,
		\end{align*}
		where $h(s)= g'(s)s+g(s)$ satisfies \eqref{hs}.
		The corresponding boundary conditions are
		\begin{align*}
			W_{x}^{\tau}\left( 0,t \right) =0,~ W_{x}^{\tau}(1,t) =-kW_{t}^{\tau}(1,t),
		\end{align*}
		and the initial data are given by
		\begin{align}\label{Wini}
			\left\{
			\begin{aligned}
				W^\tau(x,0)=W_0^\tau(x):=& V_1^\tau, \\
				W_t^\tau(x,0)=W_1^\tau(x):=& -\frac{1}{\tau}\left(V_1^\tau+f''(u_{0}^\tau)u_{0x}^\tau V_0^\tau+f'(u_0^\tau)V_{0x}^\tau\right)\\
				&+h'(u_{0x}^\tau)u_{0xx}^\tau V_{0x}^\tau+h(u_{0x}^\tau)V_{0xx}^\tau,
			\end{aligned}
			\right.
		\end{align}
		where $V_0$ and $V_1$ are defined in \eqref{Vini}.
		
		Proceeding as in Lemmas \ref{le3.1} and  \ref{le3.2} and recalling the definition of $\lambda$ in \eqref{lambda}, we obtain
		\begin{align}\label{ww}
			&\frac{d}{dt}\left( \frac{1}{2}\int_0^1{(W_{t}^{\tau}) ^2dx}+\frac{1}{2}\int_0^1{h(u_{x}^{\tau} ) (W_{x}^{\tau}) ^2dx}+\lambda \int_0^1{xW_{x}^{\tau}W_{t}^{\tau}dx} \right) +\frac{\lambda}{\tau}\int_0^1{xW_{x}^{\tau}W_{t}^{\tau}dx}\nonumber
			\\
			&\quad+\left( \frac{1}{\tau}+\frac{\lambda}{2} \right) \int_0^1{(W_{t}^{\tau}) ^2dx}+\frac{\lambda}{2}\int_0^1{h(u_{x}^{\tau})(W_{x}^{\tau}) ^2dx}+B_3(t)\nonumber
			\\
			&=\underset{K_1}{\underbrace{\frac{1}{2}\int_0^1{h' (u_{x}^{\tau} ) V_{x}^{\tau}(W_{x}^{\tau}) ^2dx}+\frac{\lambda}{2}\int_0^1{xh' (u_{x}^{\tau} ) u_{xx}^{\tau}(W_{x}^{\tau}) ^2dx}}}\nonumber
			\\
			&\quad+\underset{K_2}{\underbrace{\int_0^1{h''(u_{x}^{\tau} ) u_{xx}^{\tau}(V_{x}^{\tau} ) ^2W_{t}^{\tau}}dx+\lambda \int_0^1{xh''(u_{x}^{\tau} ) u_{xx}^{\tau}(V_{x}^{\tau} ) ^2W_{x}^{\tau}}dx}}\nonumber
			\\
			&\quad+\underset{K_3}{\underbrace{2\int_0^1{\frac{h' (u_{x}^{\tau} )}{h(u_{x}^{\tau} )}}\left[ W_{t}^{\tau}+\frac{1}{\tau}W^{\tau}-h' (u_{x}^{\tau} ) u_{xx}^{\tau}V_{x}^{\tau}+\frac{1}{\tau}\left( f''(u^{\tau} ) u_{x}^{\tau}V^{\tau}+f'(u^{\tau} ) V_{x}^{\tau} \right) \right] V_{x}^{\tau}W_{t}^{\tau}dx}}\nonumber
			\\
			&\quad\underset{K_4}{\underbrace{-2\lambda \int_0^1{x\frac{h' (u_{x}^{\tau} )}{h(u_{x}^{\tau} )}}\left[ W_{t}^{\tau}+\frac{1}{\tau}W^{\tau}-h'(u_{x}^{\tau} ) u_{xx}^{\tau}V_{x}^{\tau}+\frac{1}{\tau}\left( f''(u^{\tau} ) u_{x}^{\tau}V^{\tau}+f'(u^{\tau} ) V_{x}^{\tau} \right) \right] V_{x}^{\tau}W_{x}^{\tau}dx}}\nonumber
			\\
			&\quad\underset{K_5}{\underbrace{-\frac{1}{\tau}\int_0^1{\left[ f''' (u^{\tau})u_{x}^{\tau}(V^{\tau})^2+2f'' (u^{\tau})V^{\tau}V_{x}^{\tau}+f'' (u^{\tau})u_{x}^{\tau}W^{\tau}+f'(u^{\tau})W_{x}^{\tau} \right] W_{t}^{\tau}dx}}}\nonumber
			\\
			&\quad\underset{K_6}{\underbrace{-\frac{\lambda}{\tau}\int_0^1{x\left[ f''' (u^{\tau})u_{x}^{\tau}(V^{\tau})^2+2f'' (u^{\tau})V^{\tau}V_{x}^{\tau}+f'' (u^{\tau})u_{x}^{\tau}W^{\tau}+f'(u^{\tau})W_{x}^{\tau} \right] W_{x}^{\tau}dx}}},
		\end{align}
		where
		\[
		B_3(t) :=\left[ kh(u_{x}^{\tau}(1,t) ) -\left( \frac{\lambda}{2}+\frac{\lambda k^2}{2}h(u_{x}^{\tau}(1,t) ) \right) \right] (W_{t}^{\tau}(1,t)) ^2\ge 0,
		\]
		by \eqref{lambda} and \eqref{hs}.
		
		Now define 
		\[
		\mathcal{F}_3(t):=\frac{1}{2}\int_0^1{\left(W_{t}^{\tau} \right) ^2dx}+\frac{1}{2}\int_0^1{h(u_{x}^{\tau} ) (W_{x}^{\tau}) ^2dx}+\lambda \int_0^1{xW_{x}^{\tau}W_{t}^{\tau}dx}.
		\]
		The estimates for $K_1$--$K_6$ on the right-hand side of \eqref{ww} are similar to those in Lemma \ref{le3.1} and Lemma \ref{le3.2}. Thus, we obtain 
		\begin{align*}
			\frac{d}{dt}\mathcal{F}_3(t)+2\alpha_1\mathcal{F}_3(t)\le& C\left( \int_0^1{(V^{\tau})^2dx}+\int_0^1{(V_{x}^{\tau})^2dx}+\int_0^1{(W^{\tau})^2dx} \right) \nonumber\\
			\le &C\left( \left\| u_{0x}^{\tau} \right\| ^2+\| u_{1}^{\tau}\| ^2 \right) e^{-2\alpha_1t}+C\left( \left\| u_{0x}^{\tau} \right\| _{1}^{2}+\| u_{1}^{\tau}\| _{1}^{2} \right) e^{-2\alpha_2t},
		\end{align*}
		provided $N(t)$ is sufficiently small and $\tau\geq \tau^*$ for some $\tau^*>0$ independent of $t$ and $\tau$. 
		
		Hence, by Gr\"onwall's inequality, we have 
		\begin{align}\label{121}
			&\mathcal{F}_3(t) \nonumber\\
			\le& \mathcal{F}_3(0)e^{-2\alpha_1t}+C\left( \left\| u_{0x}^{\tau} \right\| ^2+\|u_{1}^{\tau}\|^2 \right) te^{-2\alpha_1t}+C\left( \left\| u_{0x}^{\tau} \right\| _{1}^{2}+\| u_{1}^{\tau}\| _{1}^{2} \right) e^{-2\alpha_1t}\int_0^t{e^{2\left( \alpha_1-\alpha_2 \right) s}}ds\nonumber\\
			=&\mathcal{F}_3(0)e^{-2\alpha_1t}+C\left( \left\| u_{0x}^{\tau} \right\| _{1}^{2}+\| u_{1}^{\tau}\| ^2 \right) te^{-2\alpha_1t}+\frac{C\left( \left\| u_{0x}^{\tau} \right\| _{1}^{2}+\| u_{1}^{\tau}\| _{1}^{2} \right)}{2\alpha_1-2\alpha_2}\left(e^{2\left( \alpha_1-\alpha_2 \right) t}-1 \right) e^{-2\alpha_1t}\nonumber\\
			\le &\mathcal{F}_3(0)e^{-2\alpha_1t}+C\left( \left\| u_{0x}^{\tau} \right\| _{1}^{2}+\| u_{1}^{\tau}\| ^2 \right) te^{-2\alpha_1t}+\frac{1}{2\alpha_1-2\alpha_2}C\left( \left\| u_{0x}^{\tau} \right\| _{1}^{2}+\| u_{1}^{\tau}\| _{1}^{2} \right) e^{-2\alpha_2t}\nonumber\\
			\le& \left( \frac{1}{2}\int_0^1{(W_{1}^{\tau})^2dx}+\frac{1}{2}\int_0^1{h(u_{0x}^{\tau})(W_{0x}^{\tau})^2dx}+\lambda \int_0^1{xW_{0x}^{\tau}W_{1}^{\tau}dx} \right) e^{-2\alpha_1t}\nonumber\\
			&+C\left( \left\| u_{0x}^{\tau} \right\| _{1}^{2}+\| u_{1}^{\tau}\| ^2 \right) e^{-2\alpha_2t}+\frac{1}{2\alpha_1-2\alpha_2}C\left( \left\| u_{0x}^{\tau} \right\| _{1}^{2}+\| u_{1}^{\tau}\| _{1}^{2} \right) e^{-2\alpha_2t}\nonumber\\
			\le &C\left( \| W_{0x}^{\tau}\| ^2+\| W_{1}^{\tau}\| ^2+\left\| u_{0x}^{\tau} \right\| _{1}^{2}+\| u_{1}^{\tau}\| _{1}^{2} \right) e^{-2\alpha_2t},
		\end{align}
		since $\alpha_2<\alpha_1$.
		
		As in the proofs of Lemmas \ref{le3.1} and   \ref{le3.2}, combining \eqref{121} with \eqref{Wini} and \eqref{hs}  shows that \eqref{es40} holds. This completes the proof of Lemma \ref{le3.3}.
	\end{proof}
	
	The energy estimates obtained in Lemmas \ref{le3.1}--\ref{le3.3} involve time derivatives and mixed derivatives, but do not directly provide control over the pure spatial derivatives $\|u_{xx}^{\tau}\|$, $\|u_{xxt}^{\tau}\|$ and $\|u_{xxx}^{\tau}\|$. To close the a priori estimates and complete the proof of Proposition \ref{priori}, we require the following supplementary estimates on the pure spatial derivatives.
	\begin{Lem}\label{le3.4}
		Under the assumptions of Theorem \ref{thm4}, there exist positive constants $C$ and $\tau^*$, independent of both $\tau$  and $t$, such that the following estimates hold whenever $\tau\geq\tau^*$ and $N(t)$  is sufficiently small: 
		\begin{align}\label{es4}
			\|u_{xx}^{\tau}\|\le Ce^{-\alpha_2 t}\left( \left\| u_{0x}^{\tau} \right\| _{1}+\| u_{1}^{\tau}\| _{1} \right) ,
		\end{align}
		and 
		\begin{align}\label{es5}
			\|u_{xxt}^{\tau}\|+\| u_{xxx}^{\tau}\|\le Ce^{-\alpha_2 t}\left( \left\| u_{0x}^{\tau} \right\| _{2}+\|u_{1}^{\tau}\|_{2} \right), 
		\end{align}
		where $\alpha_2$ is defined in Lemma \ref{le3.2}.
	\end{Lem}
	\begin{proof}
		Rewriting equation \eqref{sy1}, we obtain
		\begin{align}\label{71}
			u_{tt}^{\tau}+\frac{1}{\tau}u_{t}^{\tau}-h(u_{x}^{\tau})u_{xx}^{\tau}=-\frac{f'(u^{\tau})u_{x}^{\tau}}{\tau},
		\end{align}
		which yields
		\begin{align*}
			u_{xx}^{\tau}=\frac{1}{h(u_{x}^{\tau})}\left[ u_{tt}^{\tau}+\frac{1}{\tau}u_{t}^{\tau}+\frac{f'(u^{\tau})u_{x}^{\tau}}{\tau} \right].
		\end{align*}
		Hence, from Lemma \ref{le3.1} and Lemma \ref{le3.2} we immediately obtain \eqref{es4}.
		
		Differentiating \eqref{71} with respect to $t$ gives
		\begin{align*}
			u_{ttt}^{\tau}+\frac{1}{\tau}u_{tt}^{\tau}-h'(u_{x}^{\tau})u_{xt}^{\tau}u_{xx}^{\tau}-h(u_{x}^{\tau})u_{xxt}^{\tau}=-\frac{f''(u^{\tau})u_{t}^{\tau}u_{x}^{\tau}+f'(u^{\tau})u_{xt}^{\tau}}{\tau},
		\end{align*}
		and consequently,
		\begin{align}\label{72}
			u_{xxt}^{\tau}=\frac{1}{h(u_{x}^{\tau})}\left[ u_{ttt}^{\tau}+\frac{1}{\tau}u_{tt}^{\tau}-h'(u_{x}^{\tau})u_{xt}^{\tau}u_{xx}^{\tau}+\frac{f''(u^{\tau})u_{t}^{\tau}u_{x}^{\tau}+f'(u^{\tau})u_{xt}^{\tau}}{\tau} \right].
		\end{align}
		
		Finally, differentiating \eqref{71} with respect to $x$ yields
		\begin{align*}
			u_{xtt}^{\tau}+\frac{1}{\tau}u_{xt}^{\tau}-h'(u_{x}^{\tau})(u_{xx}^{\tau})^{2}-h(u_{x}^{\tau})u_{xxx}^{\tau}=-\frac{f''(u^{\tau})(u_{x}^{\tau})^{2}+f'(u^{\tau})u_{xx}^{\tau}}{\tau},
		\end{align*}
		hence
		\begin{align}\label{73}
			u_{xxx}^{\tau}=\frac{1}{h(u_{x}^{\tau})}\left[ u_{xtt}^{\tau}+\frac{1}{\tau}u_{xt}^{\tau}-h'(u_{x}^{\tau})(u_{xx}^{\tau})^{2}+\frac{f''(u^{\tau})(u_{x}^{\tau})^{2}+f'(u^{\tau})u_{xx}^{\tau}}{\tau} \right].
		\end{align}
		
		From \eqref{72} and \eqref{73}, together with Lemmas \ref{le3.1}--\ref{le3.3} and \eqref{es4}, we conclude that \eqref{es5} holds.
	\end{proof}
	
	With the a priori estimates now fully established in Lemmas \ref{le3.1}--\ref{le3.4}, we proceed to complete the proof of Theorem \ref{thm4} by combining these estimates with the local existence theory.
	
	\begin{proof}[Proof of Theorem \ref{thm4}]
		By virtue of the {\it a priori} bounds obtained in Lemmas \ref{le3.1}--\ref{le3.4}, the proof of Proposition  \ref{priori} reduces to establishing the exponential decay estimate \eqref{eest}. To this end, we first recall from Lemma \ref{le3.1} that
		\begin{align}\label{uxut}
			\|u_x^\tau\|+\|u_t^\tau\|\le Ce^{-\alpha t}(\|u_{0x}^\tau\|+\|u_1^\tau\|). 
		\end{align}
		This immediately implies
		\begin{align*}
			\left|\frac{d}{dt}\int_0^1 u^\tau(x,t)dx\right|
			\le \int_0^1 |u_t^\tau|dx
			\le \|u_t^\tau\|
			\le Ce^{-\alpha t}(\|u_{0x}^\tau\|+\|u_1^\tau\|),
		\end{align*}
		where we have used the Cauchy--Schwarz inequality. Hence the time derivative of the function $t\mapsto \int_0^1 u^\tau(x,t)dx$ is absolutely integrable, and therefore the limit
		\begin{eqnarray*}
			u^\tau_*:&=&\lim_{t\to\infty}\int_0^1 u^\tau(x,t)dx\\
			&=&\lim_{t\to\infty}\Big[\int^1_0 u^\tau_0(x) dx + \int_0^t [v^\tau(0,s)-v^\tau(1,s)]ds\Big]\\
			&=& \int^1_0 u^\tau_0(x) dx + \int_0^\infty [v^\tau(0,t)-v^\tau(1,t)]dt
		\end{eqnarray*}
		exists and is a constant. Here we used the equation $u^\tau_t+v^\tau_x=0$ and integrated it over $[0,1]\times [0,t]$ with respect to $x$ and $t$. The term $\int_0^\infty [v^\tau(0,t)-v^\tau(1,t)]dt$ represents the total net flux of the traffic flow. Furthermore, for any $t\ge0$, it holds
		\begin{align*}
			\left\|\int_0^1u^\tau(x,t)dx-u^\tau_*\right\|=&\left|\int_0^1u^\tau(x,t)dx-u^\tau_*\right|\\
			=&\left|\int_t^\infty\frac{d}{ds}\int_0^1u^\tau(x,s)dxds\right|\\
			\le& C(\|u_{0x}^\tau\|+\|u_1^\tau\|)\int_t^\infty e^{-\alpha s}ds\\
			\le & Ce^{-\alpha t}(\|u_{0x}^\tau\|+\|u_1^\tau\|).
		\end{align*}
		On the other hand, applying the Poincar\'e inequality together with \eqref{uxut} gives
		\begin{align*}
			\left\|u^\tau(\cdot,t)-\int_0^1 u^\tau(x,t)dx\right\|
			\le C\|u^\tau_x(\cdot,t)\|
			\le Ce^{-\alpha t}(\|u_{0x}^\tau\|+\|u_1^\tau\|).
		\end{align*}
		Combining the two estimates above, we obtain
		\begin{align*}
			\|u^\tau(\cdot,t)-u^*\|
			\le& \left\|u^\tau(\cdot,t)-\int_0^1 u^\tau(x,t)dx\right\|
			+\left\|\int_0^1 u^\tau(x,t)dx-u^*\right\| \\
			\le& Ce^{-\alpha t}(\|u_{0x}^\tau\|+\|u_1^\tau\|),
		\end{align*}
		which is precisely the desired estimate \eqref{eest}. Proposition \ref{priori} is therefore proved.
		
		Having established Proposition \ref{priori} and the local existence result from Proposition \ref{local}, we now proceed to prove Theorem \ref{thm4}. To this end, we assume
		\begin{align}\label{yantuo N(0)}
			N(0)\le \epsilon^*:=\min\left\{\frac{\epsilon}{2(1+C_1)},\frac{\epsilon_0}{1+C_1},\frac{\|u_0^\tau\|}{C_2}\right\}\le \epsilon_0.
		\end{align}
		Then, by Proposition \ref{local}, there exists a solution $u^\tau\in X(0,T_0)$ satisfying
		\begin{align}\label{yantuo3.54}
			N(t)\le 2N(0)\le 2\epsilon^* \le \epsilon, \quad \text{ for any } t\in[0,T_0].
		\end{align}
		Applying Proposition \ref{priori} together with \eqref{yantuo3.54}, we obtain that \eqref{est}, \eqref{eest} and \eqref{3.11} hold on $[0,T_0]$. Consequently, combining \eqref{3.11} with \eqref{yantuo N(0)} yields
		\begin{align*}
			\|u^\tau\|\le \|u_0\|+C_2N(t)\le \|u_0\|+C_2\epsilon^*\le 2\|u_0\| \quad \text{ for any } t\in[0,T_0],
		\end{align*}
		which implies that \eqref{eeest} holds on $[0,T_0]$.
		
		Now, setting $t=T_0$ in \eqref{est} and recalling \eqref{yantuo N(0)}, we obtain
		\begin{align*}
			N(T_0)\le C_1N(0)\le C_1\epsilon^*\le \epsilon_0.
		\end{align*}
		
		This, together with Proposition \ref{local} once again, allows us to extend the solution to the interval $[T_0,2T_0]$, yielding $u^\tau\in X(T_0,2T_0)$ satisfying
		\begin{align*}
			N(t)\le 2N(T_0)\le 2C_1\epsilon^* \le \epsilon, \quad \text{ for any } t\in[T_0,2T_0].
		\end{align*}
		Combining this with \eqref{yantuo3.54}, we obtain
		\begin{align*}
			N(t) \le \epsilon, \quad \text{ for any } t\in[0,2T_0].
		\end{align*}
		Applying Proposition \ref{priori} once again, we conclude that \eqref{est}, \eqref{eest}, \eqref{3.11}, and hence \eqref{eeest}, hold on $[0,2T_0]$.
		
		Hence, by repeating this continuation argument, we eventually can derive a unique global solution $u^\tau\in X(0,\infty)$ satisfying \eqref{est}, \eqref{eeest}, and \eqref{eest} for all $t\in[0,\infty)$. This completes the proof of Theorem \ref{thm4}.
	\end{proof}
	
	\subsection{Proof of Theorem \ref{thm1}}\label{sec3.3}
	
	With Theorem \ref{thm4} in hand, we are now in a position to prove our main result.
	
	\begin{proof}[Proof of Theorem \ref{thm1}]
		Although the existence of $u^\tau$ is guaranteed by Theorem \ref{thm4}, Theorem \ref{thm1} does not follow directly, since the component $v^\tau$ must be reconstructed. From \eqref{a} we see that $v^\tau$ satisfies the ODE
		\[
		v_t^\tau + \frac{1}{\tau}v^\tau = -g(u_x^\tau)u_x^\tau + \frac{f(u^\tau)}{\tau},
		\]
		with the initial condition $v^\tau(x,0)=v_0^\tau(x)$. Solving this linear ODE explicitly gives
		\begin{align}\label{vtau}
			v^\tau = e^{-\frac{t}{\tau}}v_0^\tau - \int_0^t e^{-\frac{t-s}{\tau}} g(u_x^\tau)u_x^\tau ds+ \int_0^t e^{-\frac{t-s}{\tau}} \frac{f(u^\tau)}{\tau} ds,
		\end{align}
		where $u^\tau$ is the global solution provided by Theorem \ref{thm4}. Hence the pair $(u^\tau, v^\tau)$ is a global solution to problem \eqref{a}--\eqref{b}.
		
		We next derive the required estimates for $v^\tau$. From the first equation in \eqref{a} we have $v_x^\tau = -u_t^\tau$; combining this with \eqref{est} yields
		\[
		\|v_x^\tau\|_2 = \|u_t^\tau\|_2 \le C e^{-\alpha t} (\|u_{0x}^\tau\|_2+\|u_1^\tau\|_2),
		\]
		which, together with \eqref{est}, implies \eqref{tauest}.
		
		It remains to establish the convergence estimate \eqref{tauest2}. By Theorem \ref{thm4}, we obtain
		\begin{align}\label{pr3.53}
			\|g(u_x^\tau)u_x^\tau\|\le C\|u_x^\tau\|\le Ce^{-\alpha t}(\|u_{0x}^\tau\|+\|u_1^\tau\|),
		\end{align}
		and since $f$ is smooth and hence locally Lipschitz, it follows that
		\begin{align}\label{pr3.54}
			\|f(u^\tau)-f(u^\tau_*)\|\le C\|u^\tau-u^\tau_*\|\le Ce^{-\alpha t}(\|u_{0x}^\tau\|+\|u_1^\tau\|).
		\end{align}
		Now define $v^\tau_*:=f(u^\tau_*)$, and note that
		\begin{align*}
			\frac{1}{\tau}\int_0^te^{-\frac{t-s}{\tau}}ds=1-e^{-\frac{t}{\tau}}.
		\end{align*}
		Using this fact, we may rewrite \eqref{vtau} as
		\begin{align}\label{revtau}
			v^\tau-v^\tau_*=e^{-\frac{t}{\tau}}(v_0^\tau-f(u^\tau_*))-\int_0^t e^{-\frac{t-s}{\tau}} g(u_x^\tau)u_x^\tau ds+ \int_0^t e^{-\frac{t-s}{\tau}} \frac{f(u^\tau)-f(u^\tau_*)}{\tau}ds.
		\end{align}
		Combining \eqref{pr3.53}, \eqref{pr3.54}, and \eqref{revtau}, we obtain
		\begin{align}\label{vtau-v*}
			\|v^\tau-v^\tau_*\|\le& e^{-\frac{t}{\tau}}\|v_0^\tau-f(u^\tau_*)+f(0)-f(0)\|+C\left(1+\frac{1}{\tau}\right)(\|u_{0x}^\tau\|
			+\|u_1^\tau\|)\int_0^te^{-\frac{t-s}{\tau}}e^{-\alpha s}ds\nonumber\\
			\le& Ce^{-\frac{t}{\tau}}\left(\|v_0^\tau\|+u^\tau_*+|f(0)|\right)+ C\left(1+\frac{1}{\tau}\right)(\|u_{0x}^\tau\|+\|u_0\|+\|u_1^\tau\|)\int_0^te^{-\frac{t-s}{\tau}}e^{-\alpha s}ds\nonumber\\
			\le& C\left(\|v_0^\tau\|+\|u_0^\tau\|_1+\|u_1^\tau\|+|f(0)|\right)e^{-\frac{t}{\tau}}\nonumber\\
			&+C\left(1+\frac{1}{\tau}\right)(\|u_{0x}^\tau\|+\|u_1^\tau\|)\int_0^te^{-\frac{t-s}{\tau}}e^{-\alpha s}ds,
		\end{align}
		where we have used \eqref{es2} and the fact that 
		$$u^\tau_*=\lim\limits_{t\to\infty}\int_0^1u^{\tau}dx\le C\|u_0^{\tau}\|+C(\|u_{0x}^\tau\|+\|u_1^\tau\|).$$
		A direct computation yields
		\begin{align}\label{integ}
			\int_0^te^{-\frac{t-s}{\tau}}e^{-\alpha s}ds\le
			\left\{\begin{aligned}
				&C\frac{\tau}{|1-\tau\alpha|}e^{-\min\{\alpha,\frac{1}{\tau}\}t}, \quad &\text{ if }& \tau\neq\frac{1}{\alpha},\\
				&C\frac{\alpha}{\alpha-\alpha'}e^{-\alpha't}, &\text{ if }& \tau=\frac{1}{\alpha},
			\end{aligned}\right.
		\end{align}
		for any $0<\alpha'<\alpha$. Substituting \eqref{integ} into \eqref{vtau-v*}, we obtains  
		\begin{align}\label{v-v*}
			\|v^\tau-v^\tau_*\|\le\left\{\begin{aligned}
				&C\frac{1+\tau}{|1-\tau\alpha|}\left(\|v_0^\tau\|+\|u_0^\tau\|_1+\|u_1^\tau\|+|f(0)|\right)e^{-\min\{\alpha,\frac{1}{\tau}\}t}, \quad &\text{ if }& \tau\neq\frac{1}{\alpha}\\
				&C\frac{1+\alpha}{\alpha-\alpha'}\left(\|v_0^\tau\|+\|u_0^\tau\|_1+\|u_1^\tau\|+|f(0)|\right)e^{-\alpha't}, &\text{ if }& \tau=\frac{1}{\alpha},
			\end{aligned}\right.
		\end{align}
		
		The estimate \eqref{v-v*}, together with \eqref{eest}, yields exactly \eqref{tauest2} and \eqref{tauest3}. Therefore, the proof of Theorem \ref{thm1} is complete.
	\end{proof}

	\section{Convergence in the  Large-Relaxation-Time Limit}\label{sec4}
	
	In this section, we investigate the infinite relaxation limit for the system \eqref{a}--\eqref{b}. We shall show that the solutions $(u^\tau, v^\tau)$ of \eqref{a}--\eqref{b}, obtained in Theorem \ref{thm1}, converge  to $(\hat{u}, \hat{v})$, the solutions of the nonlinear wave  equations  \eqref{hatuv-equa} established in Corollary \ref{corhat}, as $\tau \to \infty$.
	
	To facilitate the subsequent analysis, we introduce the difference functions
	\begin{align*}
		(\phi^\tau, \psi^\tau) := (u^\tau - \hat{u}, v^\tau - \hat{v}),
	\end{align*}
	which measure the deviation between the approximate solution and the limiting solution. Subtracting the  equations \eqref{hatuv-equa} from the original system \eqref{a}--\eqref{b}, we obtain the governing equations for the difference functions:
	\begin{align}\label{33}  
		\begin{cases}
			\phi_t^\tau+\psi_x^\tau=0,\\[1mm]
			\psi _{t}^{\tau}+\bigl( g(\phi _{x}^{\tau}+\hat{u}_x ) (\phi _{x}^{\tau}+\hat{u}_x ) -g(\hat{u}_x) \hat{u}_x \bigr) =\frac{f( \phi ^{\tau}+\hat{u}) -(\psi ^{\tau}+\hat{v})}{\tau}.
		\end{cases}
	\end{align}
	The initial and boundary conditions are now transformed into
	\begin{align*}
		\begin{cases}
			\phi^{\tau}(x,0)=\phi_0^{\tau}(x):=u^\tau_0(x)-\hat{u}_0(x),\\[1mm]
			\psi^{\tau}(x,0)=\psi_0^{\tau}(x):=v^\tau_0(x)-\hat{v}_0(x),
		\end{cases}
	\end{align*}
	and
	\begin{align*}
		\phi _x^{\tau} (0,t)=0, \quad  \phi _x^{\tau}(1,t)=-k\phi _t^{\tau} (1,t).
	\end{align*}
	From the first equation of \eqref{33} we obtain
	\begin{align}\label{psi.}
		\psi^{\tau}_{xt}=-\phi^{\tau}_{tt}.
	\end{align}
	Set $R(s):=sg(s)$. Differentiating the second equation of \eqref{33} with respect to $x$ gives
	\begin{align}\label{phitx}
		\psi _{tx}^{\tau}+\bigl( R(\phi _{x}^{\tau}+\hat{u}_x ) -R(\hat{u}_x) \bigr) _x = \frac{f( \phi ^{\tau}+\hat{u}) _x-(\psi ^{\tau}+\hat{v})_x}{\tau}.
	\end{align}
	Substituting \eqref{psi.} into \eqref{phitx}, we obtain a closed equation for $\phi^{\tau}$:
	\begin{align}\label{34}
		\phi^\tau _{tt}+\frac{1}{\tau}\phi^\tau_t-\bigl( R(\phi _{x}^{\tau}+\hat{u}_x ) -R(\hat{u}_x) \bigr) _x = -\frac{f(\phi ^{\tau}+\hat{u}) _x-\hat{v}_x}{\tau}.
	\end{align}
	Thus, equation \eqref{34} can be rewritten as
	\begin{align}\label{35}
		\phi^\tau _{tt}+\frac{1}{\tau}\phi^\tau_t-\bigl(R'(\hat{u}_x) \phi^\tau_x \bigr)_{x} = \frac{1}{\tau}F_x+G_{x},
	\end{align}
	where 
	\begin{align}\label{F}
		F:=-\bigl[f(\phi ^\tau+ \hat{u} ) -\hat{v}\bigr],
	\end{align}
	and 
	\begin{align}\label{G}
		G:=R(\phi^\tau_x+\hat{u}_x )-R(\hat{u}_x)-R'(\hat{u}_x)\phi^\tau_x.
	\end{align}
	By assumptions \eqref{h1} and \eqref{h2}, we have
	\begin{align}\label{rs}
		R'(s)=g'(s)s+g(s)\ge \delta>0, \quad \text{for all } s\in \mathbb{R}.
	\end{align}
	The corresponding initial and boundary data for the scalar equation \eqref{34} are
	\begin{align}\label{36}
		\begin{cases}
			\phi^\tau(x,0)=\phi^\tau_0(x),\\[0.5mm]
			\phi_t^\tau(x,0)= \phi_1^\tau(x):=-\psi_{0x}^\tau(x), \\[0.5mm]
			\phi^\tau_x(0,t)=0,\\[0.5mm]
			\phi^\tau_x(1,t)=-k\phi^\tau_t(1,t).   
		\end{cases}
	\end{align}

	Before deriving estimates for equation \eqref{35}, we first state the following observation, which follows from the estimates already established.
	
	\begin{Lem}\label{lem20}
		Assume that $\tau\ge \tau^{*}$ with $\tau^{*}>0$ as determined in Theorem \ref{thm1}, and that the initial data satisfy
		\[\|(u_{0x}^\tau,\hat{u}_{0x})\|_{2} + \| (v_{0x}^\tau,\hat{v}_{0x})\|_{2}\leq \varepsilon_3,\]
		where $0<\varepsilon_3\le \min\{\varepsilon_1,\varepsilon_2\}$ is a constant independent of $t$ and $\tau$. Then there exists a constant $C>0$, independent of $t$ and $\tau$, such that
		\begin{align}\label{uxx}
			\|u^\tau_{xx}\|_{L^\infty}+\|\hat{u}_{xt}\|_{L^\infty}+\|\hat{u}_{xx}\|_{L^\infty}+\|\phi^\tau_{xx}\|_{L^\infty}\le C\varepsilon_3.
		\end{align}
	\end{Lem}

	\begin{proof}
		Since $\varepsilon_3\le\min\{\varepsilon_1,\varepsilon_2\}$ and $\tau\ge\tau^*$, the estimates established in Corollary \ref{corhat} and Theorem \ref{thm1} imply
		\begin{align*}
			\|\hat{u}_{xt}\|_1+\|\hat{u}_{xx}\|_1 &\le C e^{-\gamma t}\bigl(\|\hat{u}_{0x}\|_2+\|\hat{v}_{0x}\|_2\bigr) \le C\varepsilon_3,\\
			\|u^{\tau}_{xx}\|_1 &\le C e^{-\sigma t}\bigl(\|u_{0x}^\tau\|_2+\|v_{0x}^\tau\|_2\bigr) \le C\varepsilon_3.
		\end{align*}
		By the Sobolev embedding $H^1\hookrightarrow L^\infty$, we obtain
		\[
		\|u^\tau_{xx}\|_{L^\infty}+\|\hat{u}_{xt}\|_{L^\infty}+\|\hat{u}_{xx}\|_{L^\infty}\le C\varepsilon_3.
		\]
		The estimate for $\phi^\tau_{xx}$ then follows from $\phi^\tau=u^\tau-\hat{u}$. This completes the proof.
	\end{proof}

	Now we derive basic estimates for $F$ and $G$ defined in \eqref{F} and \eqref{G}, which play a key role in the subsequent analysis.
	
	\begin{Lem}\label{lem21}
		Under the assumptions of Lemma \ref{lem20}, there exist constants $C>0$ and $\hat{\alpha}>0$, independent of $\tau$ and $t$, such that
		\begin{align*}
			|F_x| &\le C e^{-\hat{\alpha} t}, \qquad \text{for all } t>0,\\
			|G_x| &\le C\varepsilon_3 |\phi^\tau_x|, \qquad \text{for all } t>0,
		\end{align*}
		where $F$ and $G$ are given by \eqref{F} and \eqref{G}, respectively.
	\end{Lem}
	
	\begin{proof}
		From the definition of $F$ in \eqref{F}, a direct computation gives
		\[
		|F_x| = \bigl| -f'(\phi^\tau+\hat{u})(\phi^\tau_x+\hat{u}_x) + \hat{v}_x \bigr|
		\le C\bigl(|\phi^\tau_x|+|\hat{u}_x|+|\hat{v}_x|\bigr)
		\le C e^{-\hat{\alpha} t},
		\]
		where we have used Theorems \ref{thm1} and \ref{thm2}, the Sobolev embedding theorem, and set $\hat{\alpha} = \min\{\gamma,\sigma\}$.
		
		For $G$ defined in \eqref{G}, differentiating with respect to $x$ yields
		\begin{align*}
			G_x = \Bigl[ R'\bigl(\phi^\tau_x+\hat{u}_x\bigr) - R'(\hat{u}_x) - R''(\hat{u}_x)\phi^\tau_x \Bigr] \hat{u}_{xx} + \Bigl[ R'\bigl(\phi^\tau_x+\hat{u}_x\bigr) - R'(\hat{u}_x) \Bigr] \phi^\tau_{xx}.
		\end{align*}
		Taking absolute values and applying Taylor's formula, we obtain
		\[
		|G_x| \le C\Bigl( (\phi^\tau_x)^2 |\hat{u}_{xx}| + |\phi^\tau_x||\phi^\tau_{xx}| \Bigr).
		\]
		By Lemma \ref{lem20}, $|\phi^\tau_{xx}|$ and $|\hat{u}_{xx}|$ are bounded by $C\varepsilon_3$, and $|\phi^\tau_x|$ is bounded pointwise via the Sobolev embedding. Consequently,
		\[
		|G_x| \le C\varepsilon_3 |\phi^\tau_x|.
		\]
		This completes the proof.
	\end{proof}
	
	With the above preparations in hand, we now proceed to estimate $\phi^\tau$.
	
	\begin{Lem}\label{1p}
		Under the assumptions of Lemma \ref{lem20}, there exist positive constants $C$ and $\mu$, independent of $\tau$ and $t$, such that the estimates
		\begin{align}\label{phixt}
			\|\phi^\tau_x\|^2 + \|\phi^\tau_t\|^2 &\le \frac{C e^{-\mu t}}{\tau^2} + Ce^{-\mu t}\bigl( \|\phi^\tau_{0x}\|^2 + \|\phi^\tau_1\|^2 \bigr), \\[2mm]
			\|\phi^\tau\|^2 &\le \frac{C}{\tau^2} + C\bigl( \|\phi^\tau_0\|_1^2 + \|\phi^\tau_1\|^2 \bigr) \label{phi},
		\end{align}
		hold for all $t>0$.
	\end{Lem}
	
	\begin{proof}
		Multiplying \eqref{35} by $\phi_t^\tau$ ans integrating the resulting equation over $(0,1)$ with respect to $x$,  we obtain
		\begin{align}\label{phit}
			\frac{1}{2}\frac{d}{dt}\int_0^1(\phi_t^\tau)^2dx +\frac{1}{\tau}\int_0^1(\phi_t^\tau)^2dx -\int_0^1\bigl(R'(\hat{u}_x)\phi_x^\tau\bigr)_x\phi_t^\tau dx
			= \frac{1}{\tau}\int_0^1 F_x\phi_t^\tau dx + \int_0^1 G_x\phi_t^\tau dx .
		\end{align}
		Integrating by parts in the third term on the left-hand side of \eqref{phit} and employing the boundary condition \eqref{36} gives
		\begin{align*}
			\begin{aligned}
				-\int_0^1\bigl(R'(\hat{u}_x)\phi_x^\tau\bigr)_x\phi_t^\tau dx
				=& kR'\bigl(\hat{u}_x(1,t)\bigr)\bigl(\phi_t^\tau(1,t)\bigr)^2 + \int_0^1 R'(\hat{u}_x)\phi_x^\tau\phi_{xt}^\tau dx \\[1mm]
				=& kR'\bigl(\hat{u}_x(1,t)\bigr)\bigl(\phi_t^\tau(1,t)\bigr)^2 + \frac{1}{2}\frac{d}{dt}\int_0^1 R'(\hat{u}_x)(\phi_x^\tau)^2dx \\
				& -\frac{1}{2}\int_0^1 R''(\hat{u}_x)\hat{u}_{xt}(\phi_x^\tau)^2dx .
			\end{aligned}
		\end{align*}
		Consequently, it holds
		\begin{align}\label{p2}
			&\frac{d}{dt}\Bigl( \frac{1}{2}\int_0^1(\phi_t^\tau)^2dx + \frac{1}{2}\int_0^1 R'(\hat{u}_x)(\phi_x^\tau)^2dx \Bigr) + \frac{1}{\tau}\int_0^1(\phi_t^\tau)^2dx \nonumber\\
			&\qquad + kR'\bigl(\hat{u}_x(1,t)\bigr)\bigl(\phi_t^\tau(1,t)\bigr)^2 \nonumber\\
			&= \frac{1}{2}\int_0^1 R''(\hat{u}_x)\hat{u}_{xt}(\phi_x^\tau)^2dx + \frac{1}{\tau}\int_0^1 F_x\phi_t^\tau dx + \int_0^1 G_x\phi_t^\tau dx .
		\end{align}
		
		Now let us consider the modified energy $\int_0^1 x\phi_x^\tau\phi_t^\tau dx$. Differentiating with respect to $t$ and integrating by parts yields
		\begin{align*}
			\begin{aligned}
				\frac{d}{dt}\int_0^1 x\phi_x^\tau\phi_t^\tau dx
				=&\int_0^1 x\phi_{xt}^\tau\phi_t^\tau dx + \int_0^1 x\phi_x^\tau\phi_{tt}^\tau dx \\
				=& \frac{1}{2}\int_0^1 x\bigl((\phi_t^\tau)^2\bigr)_xdx + \int_0^1 x\phi_x^\tau\Bigl( -\frac{1}{\tau}\phi_t^\tau + \bigl(R'(\hat{u}_x)\phi_x^\tau\bigr)_x + \frac{1}{\tau}F_x + G_x \Bigr)dx \\
				=& \frac{1}{2}\bigl(\phi_t^\tau(1,t)\bigr)^2 - \frac{1}{2}\int_0^1(\phi_t^\tau)^2dx - \frac{1}{\tau}\int_0^1 x\phi_x^\tau\phi_t^\tau dx \\
				&+ \int_0^1 x\phi_x^\tau\bigl(R'(\hat{u}_x)\phi_x^\tau\bigr)_xdx + \frac{1}{\tau}\int_0^1 xF_x\phi_x^\tau dx + \int_0^1 xG_x\phi_x^\tau dx .
			\end{aligned}
		\end{align*}
		For the term involving $(R'(\hat{u}_x)\phi_x^\tau)_x$, integration by parts together with the boundary condition \eqref{36} gives
		\begin{align*}
			\begin{aligned}
				\int_0^1 x\phi_x^\tau\bigl(R'(\hat{u}_x)\phi_x^\tau\bigr)_xdx
				= &k^2 R'\bigl(\hat{u}_x(1,t)\bigr)\bigl(\phi_t^\tau(1,t)\bigr)^2 - \int_0^1 R'(\hat{u}_x)(\phi_x^\tau)^2dx \\
				& -\frac{1}{2}\int_0^1 xR'(\hat{u}_x)\bigl((\phi_x^\tau)^2\bigr)_xdx \\
				=& \frac{k^2}{2} R'\bigl(\hat{u}_x(1,t)\bigr)\bigl(\phi_t^\tau(1,t)\bigr)^2 - \frac{1}{2}\int_0^1 R'(\hat{u}_x)(\phi_x^\tau)^2dx \\
				& + \frac{1}{2}\int_0^1 xR''(\hat{u}_x)\hat{u}_{xx}(\phi_x^\tau)^2dx .
			\end{aligned}
		\end{align*}
		Inserting this identity, we obtain
		\begin{align}\label{p3}
			&\frac{d}{dt}\int_0^1 x\phi_x^\tau\phi_t^\tau dx + \frac{1}{2}\int_0^1(\phi_t^\tau)^2dx + \frac{1}{2}\int_0^1 R'(\hat{u}_x)(\phi_x^\tau)^2dx + \frac{1}{\tau}\int_0^1 x\phi_x^\tau\phi_t^\tau dx \nonumber\\
			&= \Bigl[ \frac{1}{2} + \frac{k^2}{2} R'\bigl(\hat{u}_x(1,t)\bigr) \Bigr] \bigl(\phi_t^\tau(1,t)\bigr)^2 + \frac{1}{2}\int_0^1 xR''(\hat{u}_x)\hat{u}_{xx}(\phi_x^\tau)^2dx \nonumber\\
			&\quad + \frac{1}{\tau}\int_0^1 xF_x\phi_x^\tau dx + \int_0^1 xG_x\phi_x^\tau dx .
		\end{align}
		
		Now, let $\lambda$ be as defined in \eqref{lambda}. Multiplying \eqref{p3} by $\lambda$ and adding the result to \eqref{p2}, we arrive at
		\begin{align}\label{p4}
			&\frac{d}{dt}\Bigl( \frac{1}{2}\int_0^1(\phi_t^\tau)^2dx + \frac{1}{2}\int_0^1 R'(\hat{u}_x)(\phi_x^\tau)^2dx + \lambda\int_0^1 x\phi_x^\tau\phi_t^\tau dx \Bigr) \nonumber\\
			&\quad + \frac{\lambda}{\tau}\int_0^1 x\phi_x^\tau\phi_t^\tau dx + \frac{\lambda}{2}\int_0^1(\phi_t^\tau)^2dx + \frac{\lambda}{2}\int_0^1 R'(\hat{u}_x)(\phi_x^\tau)^2dx + B(t) \nonumber\\
			&\le \underbrace{\frac{1}{2}\int_0^1 R''(\hat{u}_x)\hat{u}_{xt}(\phi_x^\tau)^2dx + \frac{\lambda}{2}\int_0^1 xR''(\hat{u}_x)\hat{u}_{xx}(\phi_x^\tau)^2dx}_{L_1} \nonumber\\
			&\quad + \underbrace{\frac{1}{\tau}\int_0^1 F_x\phi_t^\tau dx + \frac{\lambda}{\tau}\int_0^1 xF_x\phi_x^\tau dx}_{L_2} + \underbrace{\int_0^1 G_x\phi_t^\tau dx + \lambda\int_0^1 xG_x\phi_x^\tau dx}_{L_3},
		\end{align}
		where the non-negativity of 
		\[
		B(t):=\Bigl[ kR'\bigl(\hat{u}_x(1,t)\bigr) - \Bigl( \frac{\lambda}{2} + \frac{\lambda k^2}{2} R'\bigl(\hat{u}_x(1,t)\bigr) \Bigr) \Bigr] \bigl(\phi_t^\tau(1,t)\bigr)^2 \ge 0,
		\]
		follows from  \eqref{lambda} and \eqref{rs}.
		
		We now estimate $L_1$, $L_2$ and $L_3$. For $L_1$, using Lemma \ref{lem20}, \eqref{lambda} and \eqref{rs}, we obtain
		\begin{align}\label{L1}
			L_1 \le \frac{C\varepsilon_3(1+\lambda)}{\delta}\int_0^1 R'(\hat{u}_x)(\phi_x^\tau)^2dx \le \frac{\lambda}{8}\int_0^1 R'(\hat{u}_x)(\phi_x^\tau)^2dx,
		\end{align}
		provided $\varepsilon_3 \le \frac{C\lambda\delta}{1+\lambda}$.
		
		For $L_2$, we apply Lemma \ref{lem21} and Young's inequality:
		\begin{align}\label{L2}
			L_2 &\le \frac{Ce^{-\hat{\alpha}t}}{\tau}\int_0^1 |\phi_t^\tau|dx + \frac{C\lambda e^{-\hat{\alpha}t}}{\tau\sqrt{\delta}}\int_0^1 \sqrt{R'(\hat{u}_x)}|\phi_x^\tau|dx \nonumber\\
			&\le \frac{\lambda}{4}\int_0^1(\phi_t^\tau)^2dx + \frac{Ce^{-2\hat{\alpha}t}}{\tau^2\lambda} + \frac{\lambda}{8}\int_0^1 R'(\hat{u}_x)(\phi_x^\tau)^2dx + \frac{C\lambda e^{-2\hat{\alpha}t}}{\tau^2\delta} \nonumber\\
			&\le \frac{\lambda}{4}\int_0^1(\phi_t^\tau)^2dx + \frac{\lambda}{8}\int_0^1 R'(\hat{u}_x)(\phi_x^\tau)^2dx + \frac{Ce^{-2\hat{\alpha}t}}{\tau^2}.
		\end{align}
		
		For $L_3$, using Lemma \ref{lem21} together with \eqref{rs} and Young's inequality, we deduce
		\begin{align}\label{L3}
			L_3 &\le C\varepsilon_3\int_0^1 |\phi_x^\tau||\phi_t^\tau|dx + C\lambda\varepsilon_3\int_0^1 |\phi_x^\tau|^2dx \nonumber\\
			&\le \frac{C\varepsilon_3}{\sqrt{\delta}}\int_0^1 \sqrt{R'(\hat{u}_x)}|\phi_x^\tau||\phi_t^\tau|dx + \frac{C\lambda\varepsilon_3}{\delta}\int_0^1 R'(\hat{u}_x)(\phi_x^\tau)^2dx \nonumber\\
			&\le \frac{\lambda}{16}\int_0^1 R'(\hat{u}_x)(\phi_x^\tau)^2dx + \frac{C\varepsilon_3^2}{\lambda\delta}\int_0^1(\phi_t^\tau)^2dx + \frac{\lambda}{16}\int_0^1 R'(\hat{u}_x)(\phi_x^\tau)^2dx \nonumber\\
			&\le \frac{\lambda}{8}\int_0^1 R'(\hat{u}_x)(\phi_x^\tau)^2dx + \frac{\lambda}{8}\int_0^1(\phi_t^\tau)^2dx,
		\end{align}
		provided that $\varepsilon_3 $ is sufficiently small. 
		
		Now define
		\[
		\mathcal{F}(t) := \frac{1}{2}\int_0^1(\phi_t^\tau)^2dx + \frac{1}{2}\int_0^1 R'(\hat{u}_x)(\phi_x^\tau)^2dx + \lambda\int_0^1 x\phi_x^\tau\phi_t^\tau dx .
		\]
		Inserting \eqref{L1}--\eqref{L3} into \eqref{p4} yields
		\begin{align*}
			\frac{d}{dt}\mathcal{F}(t) + \frac{\lambda}{8}\int_0^1(\phi_t^\tau)^2dx + \frac{\lambda}{8}\int_0^1 R'(\hat{u}_x)(\phi_x^\tau)^2dx + \frac{\lambda}{\tau}\int_0^1 x\phi_x^\tau\phi_t^\tau dx \le \frac{Ce^{-2\hat{\alpha}t}}{\tau^2}.
		\end{align*}
		Choosing $\tau \ge \frac{\lambda}{\beta}$ with $\beta$ defined in \eqref{be} and applying Young's inequality, we deduce
		\begin{align*}
			&\frac{d}{dt}\mathcal{F}(t) + \frac{\lambda}{8}\int_0^1(\phi_t^\tau)^2dx + \frac{\lambda}{8}\int_0^1 R'(\hat{u}_x)(\phi_x^\tau)^2dx + \beta\int_0^1 x\phi_x^\tau\phi_t^\tau dx \\
			&\le \Bigl(\beta - \frac{\lambda}{\tau}\Bigr)\int_0^1 x\phi_x^\tau\phi_t^\tau dx + \frac{Ce^{-2\hat{\alpha}t}}{\tau^2} \\
			&\le \beta\int_0^1 |\phi_x^\tau||\phi_t^\tau|dx + \frac{Ce^{-2\hat{\alpha}t}}{\tau^2} \\
			&\le \frac{\beta}{2\sqrt{\delta}}\int_0^1(\phi_t^\tau)^2dx + \frac{\beta}{2\sqrt{\delta}}\int_0^1 R'(\hat{u}_x)(\phi_x^\tau)^2dx + \frac{Ce^{-2\hat{\alpha}t}}{\tau^2}.
		\end{align*}
		Rearranging and using the definition of $\mathcal{F}(t)$, we obtain
		\begin{align}\label{diffF}
			\frac{d}{dt}\mathcal{F}(t) + \alpha_1\mathcal{F}(t) \le \frac{Ce^{-2\hat{\alpha}t}}{\tau^2},
		\end{align}
		where $\alpha_1 = \frac{\lambda}{4(1+\lambda/\sqrt{\delta})}$ is exactly the constant appearing in Lemma \ref{le3.1}.
		
		Applying Gr\"onwall's inequality to \eqref{diffF} gives
		\[
		\mathcal{F}(t) \le \mathcal{F}(0)e^{-\alpha_1 t} + \frac{C}{\tau^2} e^{-\alpha_1 t}\int_0^t e^{(\alpha_1-2\hat{\alpha})s}ds .
		\]
		If $\alpha_1 \neq 2\hat{\alpha}$, the integral can be bounded by a constant times $e^{-\mu t}$ with $\mu = \min\{\alpha_1,2\hat{\alpha}\}$; the case $\alpha_1 = 2\hat{\alpha}$ only produces an additional factor $t$, which can be absorbed into the exponential at the cost of a slightly smaller exponent. In any case, there exists $\mu>0$ (depending only on $\alpha_1,\hat{\alpha}$) such that
		\begin{align}\label{Fe}
			\mathcal{F}(t) \le \mathcal{F}(0)e^{-\mu t} + \frac{Ce^{-\mu t}}{\tau^2}.
		\end{align}
		By the same arguments as in Section \ref{sec3}, using \eqref{rs} and the initial conditions \eqref{36}, we see that the estimate \eqref{Fe} directly implies \eqref{phixt}.
		
		To estimate $\phi^\tau$ itself, we write $\phi^\tau(x,t) - \phi_0^\tau(x) = \int_0^t \phi_s^\tau(x,s)ds$. Then, by \eqref{phixt},
		\begin{align*}
			\begin{aligned}
				\|\phi^\tau\|^2 &\le C\|\phi_0^\tau\|^2 + C\int_0^t \|\phi_s^\tau\|^2ds \\
				&\le C\|\phi_0^\tau\|^2 + C\int_0^t \Bigl( \frac{e^{-\mu s}}{\tau^2} + e^{-\mu s}\bigl(\|\phi_{0x}^\tau\|^2 + \|\phi_1^\tau\|^2\bigr) \Bigr)ds \\
				&\le C\|\phi_0^\tau\|^2 + \frac{C}{\mu\tau^2}(1-e^{-\mu t}) + \frac{C}{\mu}(1-e^{-\mu t})\bigl(\|\phi_{0x}^\tau\|^2 + \|\phi_1^\tau\|^2\bigr) \\
				&\le \frac{C}{\tau^2} + C\bigl(\|\phi_0^\tau\|_1^2 + \|\phi_1^\tau\|^2\bigr),
			\end{aligned}
		\end{align*}
		which is precisely \eqref{phi}. This completes the proof of Lemma \ref{1p}.
	\end{proof}

	\begin{proof}[Proof of Theorem \ref{thm3}]
		By  Lemma \ref{1p} and the Sobolev embedding theorem, we obtain
		\begin{align}\label{ph}
			\|\phi^\tau\|_{L^\infty} \le C\|\phi^\tau\|_{1} \le C\bigl(\|\phi_0^\tau\|_1 + \|\phi_1^\tau\|\bigr) + \frac{C}{\tau}.
		\end{align}
		Next, from \eqref{33} we see that $\psi^\tau$ satisfies the ODE
		\[
		\psi _{t}^{\tau}+\frac{1}{\tau}\psi ^{\tau}=-\bigl[ g(\phi _{x}^{\tau}+\hat{u}_x)(\phi _{x}^{\tau}+\hat{u}_x)-g(\hat{u}_x)\hat{u}_x \bigr] +\frac{f(\phi ^{\tau}+\hat{u})-\hat{v}}{\tau}.
		\]
		Setting $R(s)=s g(s)$, we can  rewrite this equation as
		\begin{align*}
			\psi _{t}^{\tau}+\frac{1}{\tau}\psi ^{\tau}&=-\bigl( R(\phi _{x}^{\tau}+\hat{u}_x)-R(\hat{u}_x) \bigr) +\frac{f(\phi ^{\tau}+\hat{u})-\hat{v}}{\tau}\\
			&=-\bigl( R(\phi _{x}^{\tau}+\hat{u}_x)-R(\hat{u}_x) \bigr) +\frac{f(\phi ^{\tau}+\hat{u})-f\left( \hat{u}\right)}{\tau}+\frac{f\left( \hat{u} \right)}{\tau}-\frac{\hat{v}}{\tau},
		\end{align*}
		with the initial condition $\psi^\tau(x,0)=\psi_0^\tau(x)$. Solving the linear ODE explicitly gives
		\begin{align}\label{psi}
			\psi ^{\tau}&=e^{-\frac{t}{\tau}}\psi _{0}^{\tau}\nonumber\\
			&\quad+\int_0^t e^{-\frac{t-s}{\tau}}\Bigl[ -\bigl( R(\phi _{x}^{\tau}+\hat{u}_x)-R(\hat{u}_x) \bigr) +\frac{f(\phi ^{\tau}+\hat{u})-f\left( \hat{u}\right)}{\tau}+\frac{f\left( \hat{u} \right)}{\tau}-\frac{\hat{v}}{\tau} \Bigr] ds.
		\end{align}
		
		We now estimate $\psi^\tau$ and its derivatives. From the first equation of \eqref{33} and  \eqref{phixt}, we have 
		\begin{align}\label{r.11}
			\|\psi_x^\tau\| \le \frac{C e^{-\mu t/2}}{\tau} + C e^{-\mu t/2}\bigl(\|\phi_{0x}^\tau\| + \|\phi_1^\tau\|\bigr).
		\end{align}

		Moreover, using Corollary \ref{corhat}, Lemma \ref{1p} and \eqref{h3},  we obtain 
		\begin{align*}
			\|R(\phi_x^\tau+\hat{u}_x)-R(\hat{u}_x)\| &\le C\|\phi_x^\tau\| \le \frac{C e^{-\mu t/2}}{\tau} + C e^{-\mu t/2}\left( \|\phi_{0x}^\tau\| + \|\phi_1^\tau\|\right),\\
			\Bigl\| \frac{f(\phi ^{\tau}+ \hat{u})-f\left(  \hat{u}\right)}{\tau} \Bigr\| &\le \frac{C}{\tau}\|\phi ^{\tau}\| \le \frac{C}{\tau ^2}+\frac{C}{\tau}\left( \|\phi _{0}^{\tau}\| _1+\| \phi _{1}^{\tau}\| \right) ,\\
			\Bigl\| \frac{f\left(  \hat{u} \right)}{\tau} \Bigr\| =\Bigl\| \frac{f\left(  \hat{u}\right) -f\left( 0 \right)}{\tau} \Bigr\| &\le \frac{C}{\tau}\left\|  \hat{u} \right\| \le \frac{Ce^{-\mu t/2}}{\tau},\\
			\left\| \frac{ \hat{v}}{\tau} \right\| &\le \frac{1}{\tau}\left\|  \hat{v} \right\| \le \frac{Ce^{-\mu t/2}}{\tau}.
		\end{align*}
		Inserting these estimates into \eqref{psi} yields
		\begin{align*}
			\| \psi ^{\tau}\| \le &\| \psi _{0}^{\tau}\| +C\int_0^t{e^{-\frac{t-s}{\tau}}\Bigl( \frac{e^{-\mu s/2}}{\tau}+e^{-\mu s/2}(\| \phi _{0x}^{\tau}\| +\| \phi _{1}^{\tau}\| ) \Bigr) ds}\\
			&+C\int_0^t{e^{-\frac{t-s}{\tau}}\left( \frac{1}{\tau ^2}+\frac{1}{\tau}\left( \| \phi _{0}^{\tau}\| _1+\| \phi _{1}^{\tau}\| \right) \right) ds}\\
			&+\frac{C}{\tau}\int_0^t{e^{-\frac{t-s}{\tau}}e^{-\mu s/2}ds}\\
			\le &\| \psi _{0}^{\tau}\| +C\left( \frac{1}{\tau}+(\| \phi _{0x}^{\tau}\| +\| \phi _{1}^{\tau}\|) \right) \int_0^t{e^{-\frac{t-s}{\tau}}e^{-\mu s/2}ds}\\
			&+C\left( \frac{1}{\tau ^2}+\frac{1}{\tau}(\| \phi _{0}^{\tau}\| _1+\| \phi _{1}^{\tau})\| ) \right) \tau \left( 1-e^{-\frac{t}{\tau}} \right)\\
			&+\frac{C}{\tau}\int_0^t{e^{-\frac{t-s}{\tau}}e^{-\mu s/2}ds}\\
			\le &C(\bigl\| \psi _{0}^{\tau}\| +\| \phi _{0}^{\tau}\| _1+\| \phi _{1}^{\tau}\| \bigr) +\frac{C}{\tau}.
		\end{align*}
		
		Combining this with the estimate \eqref{r.11} for $\psi_x^\tau$  and applying the Sobolev embedding theorem gives
		\begin{align}\label{ps}
			\|\psi^\tau\|_{L^\infty} \le C\|\psi^\tau\|_{1} \le C\bigl(\|\psi_0^\tau\|_1 + \|\phi_0^\tau\|_1 + \|\phi_{1}^\tau\|\bigr) + \frac{C}{\tau}.
		\end{align}
		
		From \eqref{ph} and \eqref{ps}, we immediately obtain \eqref{goal1}. If, in addition, the initial perturbations satisfy
		\[
		\|\phi_0^\tau\|_1 + \|\psi_0^\tau\|_1 \le \frac{C}{\tau},
		\]
		then the solutions of \eqref{a}--\eqref{b} converge to those  of \eqref{hatuv-equa} as $\tau\to+\infty$ in the sense stated in \eqref{goal2} for any $t>0$.
	\end{proof}
	\section{Numerical Simulations}\label{sec5}
In this section, we present numerical simulations to illustrate and validate the main theoretical findings established in
	 Theorem~\ref{thm3}. To this end, we specify the flux function $f(s)$ and the gradient sensitivity function $g(s)$ in system \eqref{a} with \eqref{b} as follows:
	\[f(s)=4s(1-s),~g(s)=s^2+1.\]
	The initial density and flow profiles $(u^\tau_0(x), v^\tau_0(x))$ are prescribed by
	\begin{align*}
		\begin{cases}
			u_{0}^{\tau}(x)=0.62+0.18B(x)\frac{\cos \theta +0.08\cos\mathrm{(}2\theta )}{1.08}+\frac{1}{\tau},\\[0.5em]
			v_{0}^{\tau}(x)=\frac{f\left( u_0(x) \right) \exp \left[ 0.65B(x)\sin \left( \theta -\frac{\pi}{2} \right) \right]}{1-f\left( u_0(x) \right) +f\left( u_0(x) \right) \exp \left[ 0.65B(x)\sin \left( \theta -\frac{\pi}{2} \right) \right]},\\
		\end{cases}
	\end{align*}
	where $\theta=4\pi(x-0.5),~B(x)=b\!\left(\frac{x-0.5}{0.48}\right)$, 
	with
	\[b(z)=\begin{cases}
		\displaystyle
		\exp\!\left(1-\frac{1}{1-z^2}\right),
		& |z|<1,\\[2mm]
		0,
		& |z|\geq 1.
	\end{cases}\]
	These benchmark setups are retained consistently throughout the subsequent numerical experiments.

The initial data prescribed above incorporate a localized spatial disturbance. In the regime of a large relaxation time $\tau$, the relaxation source term $\frac{f(u^\tau) - v^\tau}{\tau}$ becomes exceedingly weak. Consequently, $v^\tau$ adapts sluggishly toward its local equilibrium profile $f(u^\tau)$, allowing the localized disturbance to persist and evolve into alternating high-- and low--density band--a signature of the classic stop-and-go traffic instability. Under the action of the proposed intelligent-control boundary, this instability is effectively mitigated. Figure~\ref{fig:stop-go-schematic} provides a schematic overview of the underlying physical mechanism:

\begin{itemize} 
\item Free flow stage ($t = t_0$): Vehicles enter from the upstream on-ramp and move smoothly downstream, forming an initially uniform stream.
\item Delayed response regime ($t = t_1$): Due to sluggish driver adaptation under a large relaxation time $\tau$, a localized perturbation grows into a stop-and-go wave structure.
\item Suppression via intelligent control: As density accumulates near the exit boundary, the adaptive damping condition senses the real-time temporal variation and dynamically adjusts the outflow, gradually absorbing the kinetic energy of the density waves and restoring traffic stability.
\end{itemize}
	
	\begin{figure}[H]
		\centering
		\includegraphics[width=0.85\linewidth]{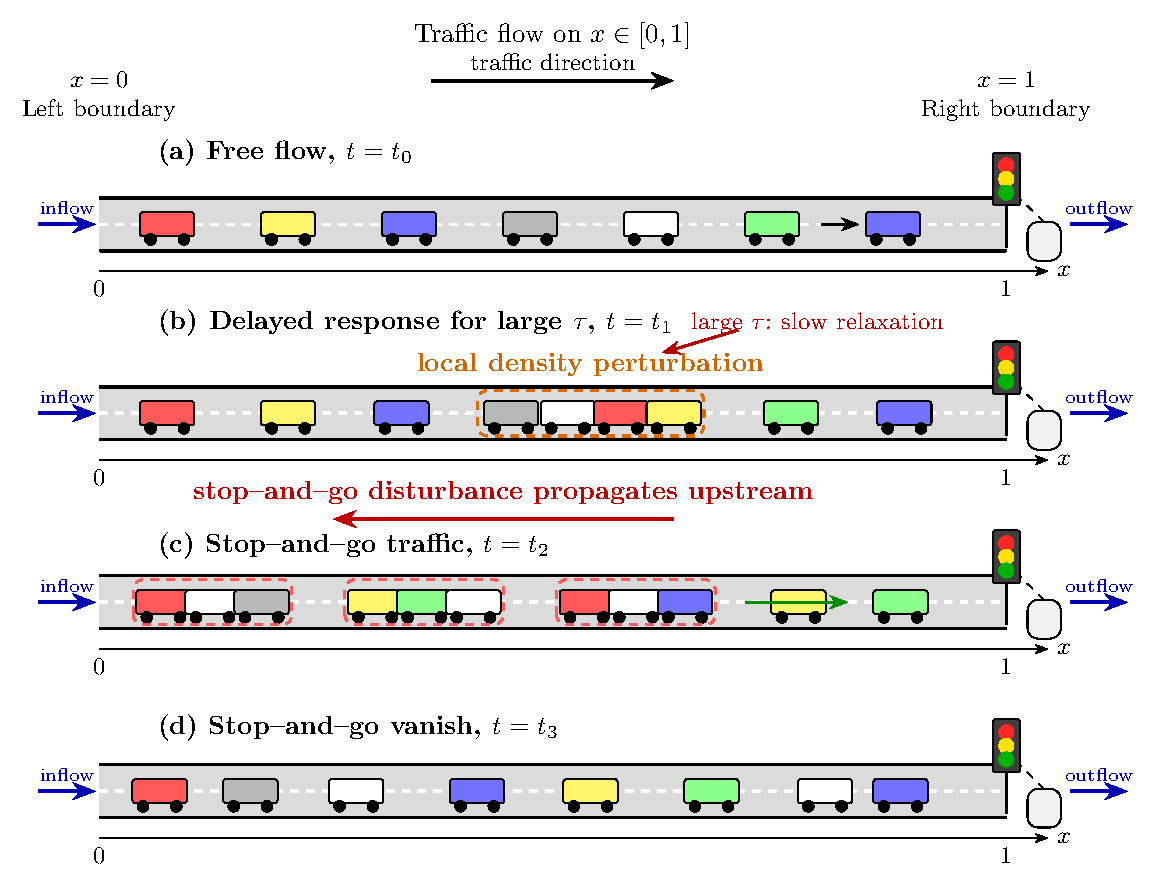}
		\caption{Schematic illustration of the formation of stop-and-go traffic induced by a large relaxation time.}
		\label{fig:stop-go-schematic}
	\end{figure}
	
	First, Theorem \ref{thm3} tells us that the  solution to the problem \eqref{a} with \eqref{b} uniformly con-
	verges toward the one to the problem \eqref{hatuv-equa} as the relaxation time $\tau$ tends to $\infty$. This result holds for the  damping boundary data 
	\[u_x^\tau(0,t)=0,~~u_x^\tau(1,t) = -ku_t^\tau(1,t),\]
	where we choose $k=0.12$.
	
	Next, to numerically verify the theoretical results established in Theorem \ref{thm3},  in this section, we use a fifth-order
	linear upwind-biased finite-difference method with global
	Lax--Friedrichs stabilization and  Kreiss--Oliger-type dissipation. We introduce the uniform nodal grid $x_i=i\Delta x,~ i=0,1,\ldots,N,~\Delta x=\frac{1}{N}$, with $N=401$. Here, we compute the spatial derivatives $(u_x^\tau(t,x_i),v_x^\tau(t,x_i))$ by the sixth-order centered difference, and all nonlinear operations, namely $\frac{f (u^\tau)-v^\tau}{\tau}$ is  evaluated pointwise on
	the spatial grid.  Let $W_h^\tau=\left\{
	\left(u_i^\tau,v_i^\tau\right)^{T}
	\right\}_{i=0}^{N}$ , then the semi-discrete operator of the system \eqref{a} is defined by
	\begin{align*}
		\left[\mathcal L_h^\tau(W_h^\tau)\right]_i
		=
		\begin{pmatrix}
			-\dfrac{1}{2}(D_h^-+D_h^+)v_i^\tau
			-\dfrac{\alpha}{2}(D_h^--D_h^+)u_i^\tau
			+Q(u_i^\tau)
			\\[3mm]
			-g(p_i^\tau)p_i^\tau
			-\dfrac{\alpha}{2}(D_h^--D_h^+)v_i^\tau
			+\dfrac{f(u_i^\tau)-v_i^\tau}{\tau}
			+Q(v_i^\tau)
		\end{pmatrix},
	\end{align*}
	where $D_h^-$ and $D_h^+$ are both the fifth-order left and
	right biased finite-difference operators, $\alpha:=
	\max_{0\leq i\leq N}\sqrt{g(p_i^\tau)}$, $p_i^\tau=u_x^\tau(t,x_i)+O(\Delta x^6)$ is the common global
	Lax--Friedrichs parameter, and  $Q$ is the optional sixth-order Kreiss--Oliger dissipation operator
	$$
	Q(q_i^\tau)=\frac{\varepsilon_{\mathrm{KO}}}{64\Delta x}
	\left(
	q_{i-3}^\tau
	-6q_{i-2}^\tau
	+15q_{i-1}^\tau
	-20q_i^\tau
	+15q_{i+1}^\tau
	-6q_{i+2}^\tau
	+q_{i+3}^\tau
	\right),
	$$
	Here, $\varepsilon_{\mathrm{KO}}$ denotes the numerical dissipation parameter.  Let \(\mathcal P_h\) denote the linear boundary-value correction
	operator. It leaves the interior grid values unchanged and modifies
	only the boundary values so that the discrete boundary conditions \eqref{b} are satisfied using sixth-order one-sided differences. This results in the following generic Euler step,
	$$
	\mathcal{E}_{\Delta t}^{\tau}[W_h^\tau]:=W_h^{\tau,n+1}=\mathcal{P}_h\left[W_h^\tau+\Delta t\left[\mathcal L_h^\tau(W_h^\tau)\right]_i\right],
	$$
	where the time step $\Delta t=\min
	\left\{0.35\frac{\Delta x}{\alpha},
	\;0.8\min_{\tau<\infty}\tau
	\right\}$ for all simulations. The Euler step above is a first-order method. In order to get a higher-order convergence rate in time, we bootstrap the Euler step above into a third-order Runge–Kutta (Strong Stability Preserving) method,
	$$
	W_h^{\tau,n+1}=
	\frac{1}{3}W_h^{\tau,n}
	+
	\frac{2}{3}
	\mathcal{E}_{\Delta t}^{\tau}
	\left[
	\frac{3}{4}W_h^{\tau,n}
	+
	\frac{1}{4}
	\mathcal{E}_{\Delta t}^{\tau}
	\left[
	\mathcal{E}_{\Delta t}^{\tau}
	\left(W_h^{\tau,n}\right)
	\right]
	\right].
	$$
	
	Finally, we consider the three numerical cases of the time $t$ as follows:
	
	Case1: $t=1$; see Figure \ref{fig1}.
	\begin{figure}[H]
		\centering
		\begin{subfigure}[b]{0.48\textwidth}
			\centering
			\includegraphics[width=3.15in]{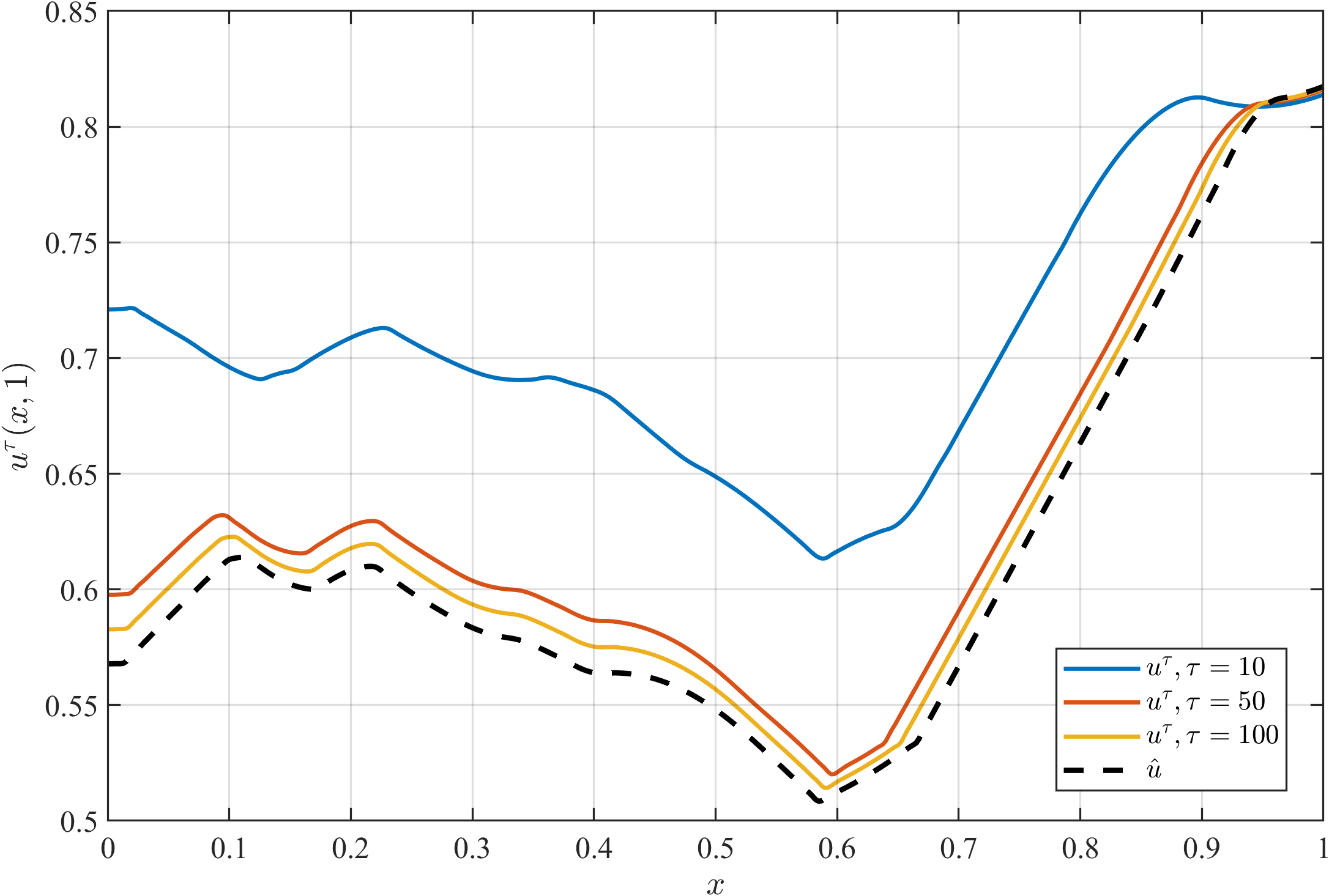}
		\end{subfigure}
		\hfill
		\begin{subfigure}[b]{0.48\textwidth}
			\centering
			\includegraphics[width=3.15in]{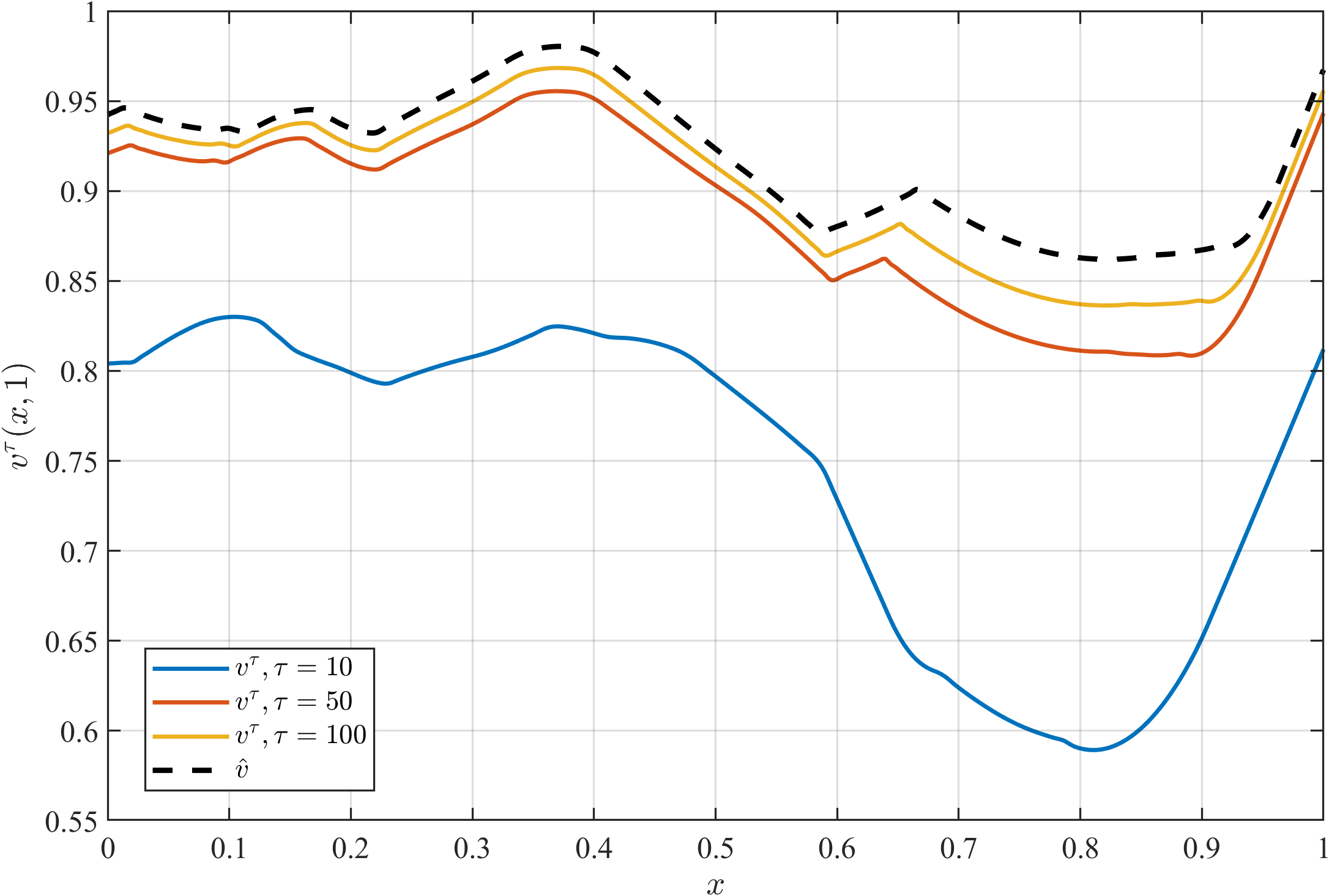}
		\end{subfigure}
		\caption{Case 1 for $(u^{\tau},v^\tau)\to (\hat{u},\hat{v})$ as $\tau\to \infty$ at $t=1$.}
		\label{fig1}
	\end{figure}

	Case 2: $t=10$; see Figure \ref{fig2}.
	
	\begin{figure}[H]
		\centering
		\begin{subfigure}[b]{0.48\textwidth}
			\centering
			\includegraphics[width=3.15in]{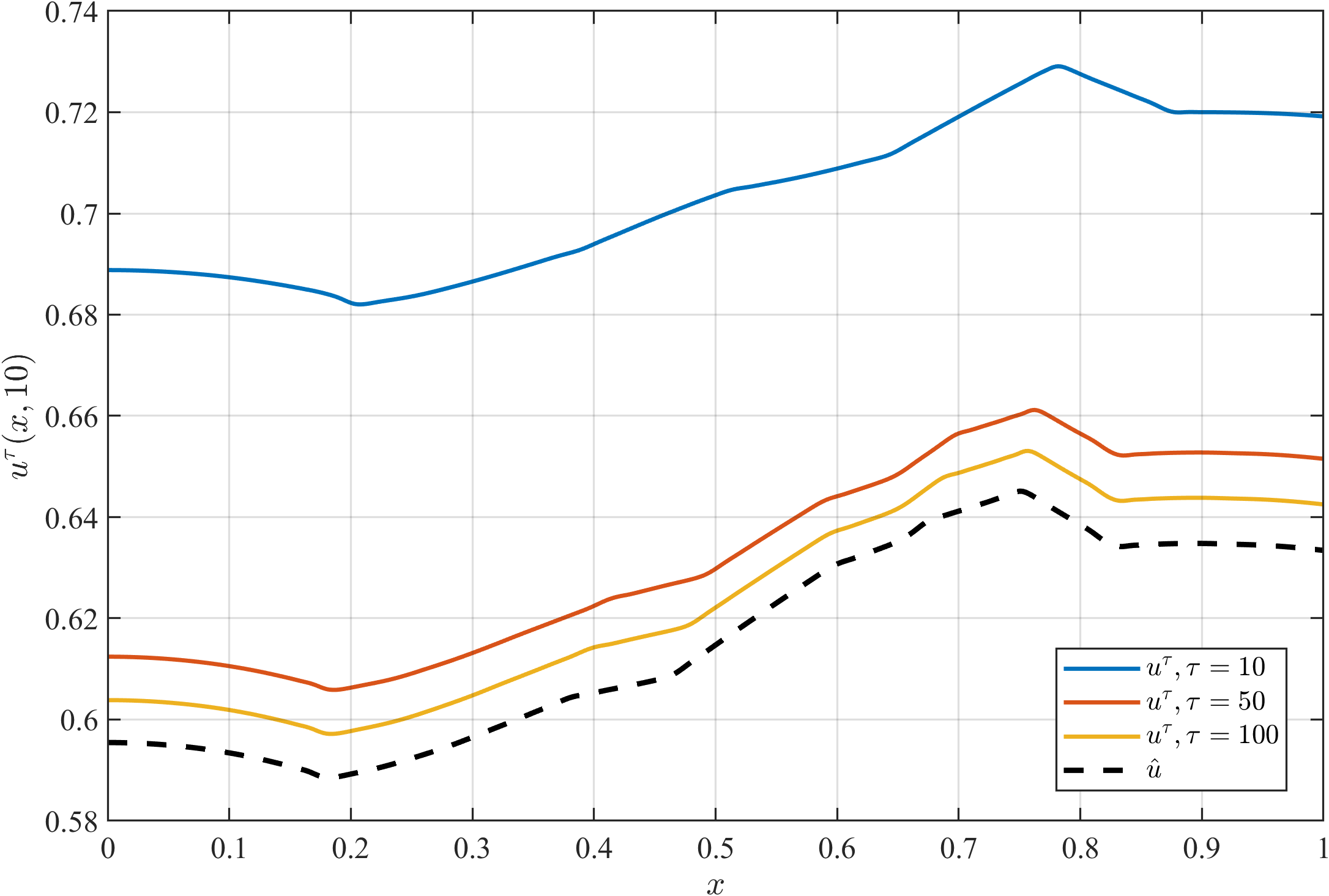}
		\end{subfigure}
		\hfill
		\begin{subfigure}[b]{0.48\textwidth}
			\centering
			\includegraphics[width=3.15in]{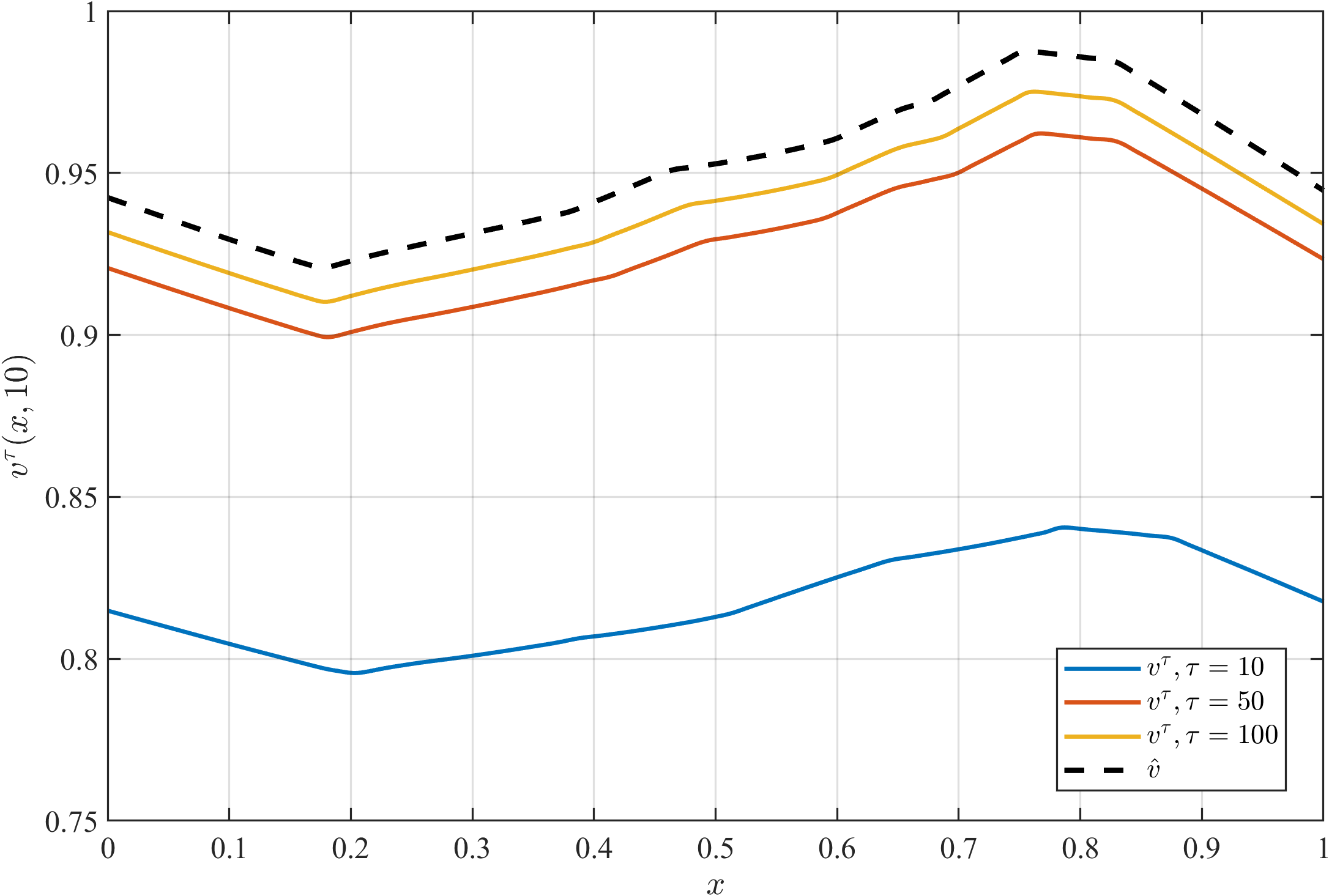}
		\end{subfigure}
		\caption{Case 2 for $(u^{\tau},v^\tau)\to (\hat{u},\hat{v})$ as $\tau\to \infty$ at $t=10$.}
		\label{fig2}
	\end{figure}
	
	Case 3: $t=100$; see Figure \ref{fig3}.
	
	\begin{figure}[H]
		\centering
		\begin{subfigure}[b]{0.48\textwidth}
			\centering
			\includegraphics[width=3.15in]{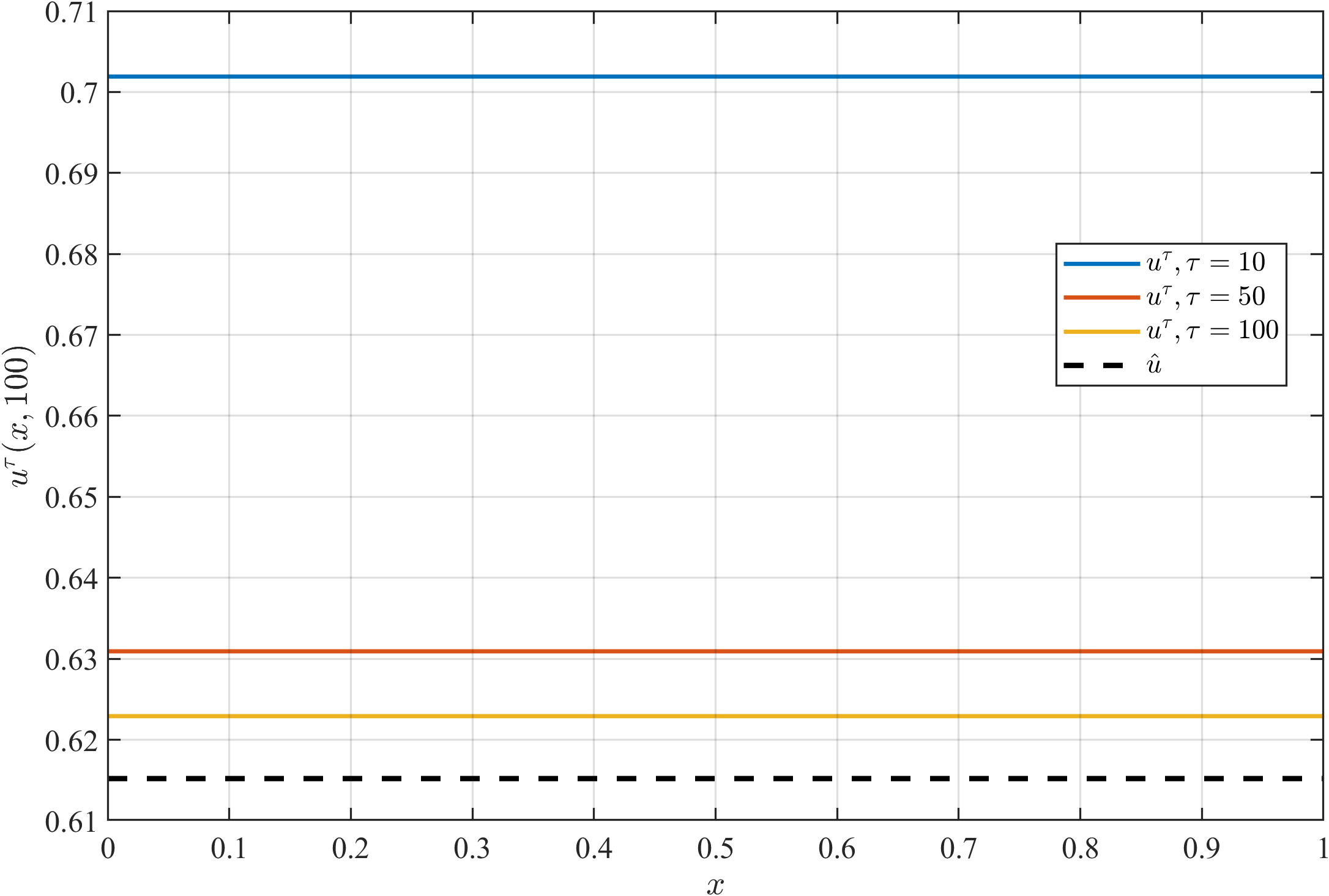}
		\end{subfigure}
		\hfill
		\begin{subfigure}[b]{0.48\textwidth}
			\centering
			\includegraphics[width=3.2in]{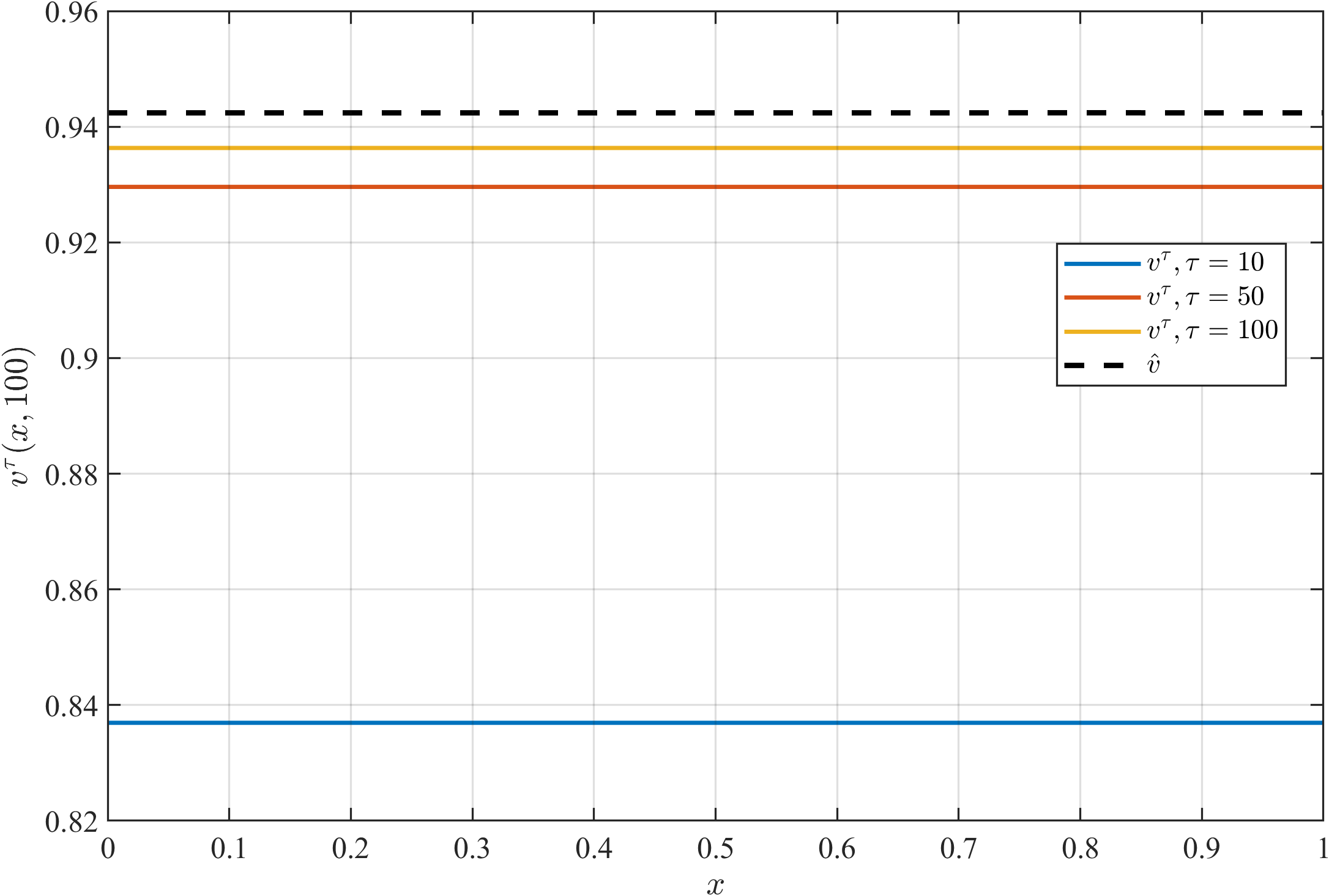}
		\end{subfigure}
		\caption{Case 3 for $(u^{\tau},v^\tau)\to (\hat{u},\hat{v})$ as $\tau\to \infty$ at $t=100$.}
		\label{fig3}
	\end{figure}

	In all the cases above, we choose the finite approximation sequence of relaxation times as $\tau_1=10,\tau_2=50,\tau_3=100$, with the limiting relaxation time $\tau = +\infty$. It can be observed from Figures \ref{fig1}--\ref{fig3} that, as \(\tau\) increases, the  solution \((u^\tau,v^\tau)\) progressively approaches the limiting solution \((\hat u,\hat v)\). In particular, the solution corresponding to \(\tau=100\) is already close to the asymptotic profile. Moreover, as shown in Figure \ref{fig1}, at \(t=1\), the  solutions $(u^\tau,v^{\tau})$ exhibit pronounced spatial variations, characterized by alternating local peaks and troughs in both the vehicle density and vehicle flow rate profiles. Such nonuniform oscillatory structures are characteristic of typical stop-and-go traffic, reflecting the repeated formation and dissipation of locally congested regions.  At \(t=10\), as illustrated in Figure \ref{fig2}, the amplitudes of these spatial oscillations are significantly reduced, and the profiles of both the vehicle density and vehicle flow rate become noticeably smoother, indicating that the stop-and-go waves are gradually becoming more gentle. At  \(t=100\), as shown in Figure \ref{fig3}, the spatial oscillations have essentially disappeared, and both the vehicle density and vehicle flow rate become nearly spatially homogeneous, approaching spatially constant equilibrium states.

	The numerical simulations reveal  two main observations. First, the stop-and-go oscillations gradually diminish over time and eventually vanish, indicating the pronounced stabilizing effect of the intelligent-control boundary, which effectively suppresses such oscillations and promotes the stabilization of the traffic flow. Second, to the best of our knowledge, the formation and suppression of stop-and-go waves in traffic flow models equipped with intelligent-control boundary have not been investigated previously, thereby highlighting a novel feature of the proposed model. All these numerical simulations conducted in this section fully corroborate the theoretical results established in Theorem \ref{thm3}.

	\section{Appendix}
	
	We now provide a brief proof of Theorem \ref{thm2}. To this end, we rewrite \eqref{2a}-\eqref{bb} as the following $p$-Laplacian type nonlinear wave equation:
	\begin{align}\label{A.model}
		\left\{
		\begin{aligned}
			&\bar{u}_{tt}-\left(g(\bar{u}_x)\bar{u}_x\right)_x=0, \quad &x&\in(0,1),\; t>0,\\
			&\bar{u}(x,0)=\bar{u}_0(x),\; \bar{u}_t(x,0)=\bar{u}_1(x), &x&\in(0,1),\\
			&\bar{u}_x(0,t)=0,\; \bar{u}_x(1,t)=-k\bar{u}_t(1,t), &t&>0,
		\end{aligned}
		\right.
	\end{align}
	where $k>0$ and $\bar{u}_1(x):= -\bar{v}_{0x}(x)$. Therefore, to prove Theorem \ref{thm2}, it suffices to consider model \eqref{A.model}.
	
	Similarly to the argument in Section \ref{sec3.1}, we introduce the solution space
	$$
	\bar{X}(0,T) := \left\{ \bar{u}(x,t) : \bar{u}\in C^0([0,T);H^3(0,1)),\; \bar{u}_t\in C^0([0,T);H^2(0,1)) \right\},
	$$
	with $0<T\le+\infty$. 
	
	More specifically, as in the argument in Section \ref{sec3.1}, we need to prove the following proposition.
	
	\begin{Pro}\label{A.pro}
		Let $(\bar{u}_0,\bar{u}_1)\in H^3(0,1)\times H^2(0,1) $. There exists a positive constant $\bar\epsilon$ such that if $\|\bar{u}_{0x}\|_2+\|\bar{u}_1\|_2\le\bar\epsilon$, then problem \eqref{A.model} admits a unique global solution $\bar{u}\in X(0,+\infty)$. Moreover, there exist constants $C>0$ and $\bar\alpha>0$ independent of $t$ such that 
		\begin{align*}
			\|\bar{u}_x\|_2+\|\bar{u}_t\|_2\le Ce^{-\bar\alpha t}\bigl(\|\bar{u}_{0x}\|_2+\|\bar{u}_1\|_2\bigr),
		\end{align*}
		and
		\begin{align*}
			\|\bar{u}\|\le 2\|\bar u_0\|
		\end{align*}
		for any $t>0$.
	\end{Pro}
	
	To prove Proposition \ref{A.pro}, we employ a standard continuation argument based on the interplay between local existence and a priori estimate. To this end, we define
	\[
	\bar{N}(t):=\sup_{0\le s\le t}\left\{e^{\bar{\alpha} s}\|\bar{u}_x(s)\|_2+e^{\bar{\alpha}s}\|\bar{u}_t(s)\|_2\right\},
	\]
	Then, by the Sobolev embedding theorem, we have
	\begin{align}\label{6.2}
		\sup _{s \in[0, t]}\left\{\left\|\bar{u}_x(s)\right\|_{L^{\infty}}+\left\|\bar{u}_{x x}(s)\right\|_{L^{\infty}}+\left\|\bar{u}_s(s)\right\|_{L^{\infty}}+\left\|\bar{u}_{s x}(s)\right\|_{L^{\infty}}\right\} \leq C \bar{N}(t) .
	\end{align}
	
	Moreover, for $\bar{u}$ itself, we obtain
	$$\left\|\bar{u}\right\| \leq\left\|\bar{u}_0\right\|+\int_0^t\left\|\bar{u}_s\right\| d s \leq\left\|\bar{u}_0\right\|+\frac{\bar{N}(t)}{\bar{\alpha}}$$
	Thus, by the Sobolev embedding theorem again
	$$\sup_{s \in[0, t]}\|\bar{u}(s)\|_{L^{\infty}} \leq\left\|\bar{u}_0\right\|+C \bar{N}(t).$$
	
	\begin{proof}[Proof of  Proposition \ref{A.pro}]
		Similar to the argument in Section \ref{sec3},   Proposition \ref{A.pro} can be proved by a continuity argument based on local existence and a priori estimates. Therefore, it suffices to establish the a priori estimate.
		
		We introduce the energy functional
		\begin{align*}
			E_1(t):=\frac12\int_0^1\bigl(\bar{u}_t^2+g(\bar{u}_x)\bar{u}_x^2\bigr)dx,
		\end{align*}
		and the modified energy
		\begin{align*}
			F_1(t):=\bar\lambda\int_0^1 x\bar{u}_x\bar{u}_tdx,
		\end{align*}
		where $\bar\lambda>0$ is a sufficiently small constant to be determined later. Using \eqref{A.model} and \eqref{6.2}, we obtain
		\begin{align}\label{6.8}
			\frac{d}{dt}E_1(t)=&\int_0^1\bar{u}_t\bar{u}_{tt}dx+\int_0^1g(\bar{u}_x)\bar{u}_x\bar{u}_{xt}dx+\frac12\int_0^1g'(\bar{u}_x)\bar{u}_{xx}\bar{u}_x^2dx\nonumber\\
			=&\int_0^1\bar{u}_t\bigl(g(\bar{u}_x)\bar{u}_x\bigr)_xdx+\int_0^1g(\bar{u}_x)\bar{u}_x\bar{u}_{xt}dx+\frac12\int_0^1g'(\bar{u}_x)\bar{u}_{xx}\bar{u}_x^2dx\nonumber\\
			=&-k g(\bar{u}_x(1,t))\bar{u}_t^2(1,t)+\frac12\int_0^1g'(\bar{u}_x)\bar{u}_{xx}\bar{u}_x^2dx\nonumber\\
			\le &- k g(\bar{u}_x(1,t))\bar{u}_t^2(1,t)+\frac{C\bar{N}(t)}{\delta}\int_0^1g(\bar{u}_x)\bar{u}_x^2dx,
		\end{align}
		and
		\begin{align}\label{6.9}
			\frac{d}{dt}F_1(t)=&\bar\lambda\int_0^1x\bar{u}_{xt}\bar{u}_tdx+\bar\lambda\int_0^1x\bar{u}_x\bar{u}_{tt}dx\nonumber\\
			=&\frac{\bar\lambda}{2}\int_0^1x(\bar{u}_t^2)_xdx+\bar\lambda\int_0^1x\bar{u}_x\bigl(g(\bar{u}_x)\bar{u}_x\bigr)_xdx\nonumber\\
			=&\frac{\bar\lambda}{2}\bar{u}_t^2(1,t)-\frac{\bar\lambda}{2}\int_0^1\bar{u}_t^2dx+\bar\lambda\int_0^1xg'(\bar{u}_x)\bar{u}_{xx}\bar{u}_x^2dx+\frac{\bar\lambda}{2}\int_0^1xg(\bar{u}_x)(\bar{u}_x^2)_xdx\nonumber\\
			=&\frac{\bar\lambda}{2}\bigl(1+k^2g(\bar{u}_x(1,t))\bigr)\bar{u}_t^2(1,t)-\frac{\bar\lambda}{2}\int_0^1\bigl(\bar{u}_t^2+g(\bar{u}_x)\bar{u}_x^2\bigr)dx+\frac{\bar\lambda}{2}\int_0^1xg'(\bar{u}_x)\bar{u}_{xx}\bar{u}_x^2dx\nonumber\\
			\le&\frac{\bar\lambda}{2}\bigl(1+k^2g(\bar{u}_x(1,t))\bigr)\bar{u}_t^2(1,t)-\frac{\bar\lambda}{2}E_1(t)+\frac{C\bar\lambda\bar{N}(t)}{\delta}\int_0^1g(\bar{u}_x)\bar{u}_x^2dx,
		\end{align}
		where we have used hypothesis \eqref{h1}. Adding \eqref{6.8} and \eqref{6.9} gives
		\begin{align*}
			&\frac{d}{dt}\bigl(E_1(t)+F_1(t)\bigr)+\frac{\bar\lambda}{2}E_1(t)+\left(kg(\bar{u}_x(1,t))-\frac{\bar\lambda}{2}\bigl(1+k^2g(\bar{u}_x(1,t))\bigr)\right)\bar{u}_t^2(1,t)\\
			&\qquad\le \frac{C(\bar\lambda+1)\bar{N}(t)}{\delta}\int_0^1g(\bar{u}_x)\bar{u}_x^2dx
			\le\frac{\bar\lambda}{4}E_1(t),
		\end{align*}
		provided $\bar{N}(t)$ is sufficiently small. Using \eqref{h1} again,
		\begin{align*}
			F_1(t)=\bar\lambda\int_0^1x\bar{u}_x\bar{u}_tdx\le\frac{\bar\lambda}{\sqrt{\delta}}\int_0^1\sqrt{g(\bar{u}_x)}|\bar{u}_x||\bar{u}_t|dx\le\frac{\bar\lambda}{2\sqrt{\delta}}E_1(t).
		\end{align*}
		Combining the above estimates and choosing
		$$0<\bar\lambda\le\frac{2k\delta}{1+k^2\delta}\le\frac{2kg(\bar{u}_x(1,t))}{1+k^2g(\bar{u}_x(1,t))},$$
		we obtain
		\begin{align*}
			\frac{d}{dt}\bigl(E_1(t)+F_1(t)\bigr)+\frac{\bar\lambda}{4}E_1(t)+\frac{\bar\lambda}{2\bigl(2+\frac{\bar\lambda}{\sqrt{\delta}}\bigr)}F_1(t)
			\le\frac{\bar\lambda}{2\bigl(2+\frac{\bar\lambda}{\sqrt{\delta}}\bigr)}F_1(t)
			\le\frac{\bar\lambda^2}{4\bigl(2\sqrt{\delta}+\bar\lambda\bigr)}E_1(t).
		\end{align*}
		Consequently,
		\begin{align}\label{6.10}
			\frac{d}{dt}\bigl(E_1(t)+F_1(t)\bigr)+\frac{\bar\lambda}{2\bigl(2+\frac{\bar\lambda}{\sqrt{\delta}}\bigr)}\bigl(E_1(t)+F_1(t)\bigr)\le0,
		\end{align}
		since $\frac{\bar\lambda}{4}=\frac{\bar\lambda}{2(2+\frac{\bar\lambda}{\sqrt{\delta}})}+\frac{\bar\lambda^2}{4(2\sqrt{\delta}+\bar\lambda)}$. Applying Gr\"nwall's inequality to \eqref{6.10} yields
		\begin{align*}
			E_1(t)+F_1(t)\le e^{-\alpha^*t}\bigl(E_1(0)+F_1(0)\bigr)\le\Bigl(1+\frac{\bar\lambda}{2\sqrt{\delta}}\Bigr)e^{-\alpha^*t}E_1(0),
		\end{align*}
		where $\alpha^*:=\frac{\bar\lambda}{2(2+\frac{\bar\lambda}{\sqrt{\delta}})}$. Hence,
		\begin{align*}
			E_1(t)\le Ce^{-\alpha^*t}E_1(0)-F_1(t)\le Ce^{-\alpha^*t}E_1(0)+\frac{\bar\lambda}{2\sqrt{\delta}}E_1(t),
		\end{align*}
		which implies
		\begin{align*}
			\Bigl(1-\frac{\bar\lambda}{2\sqrt{\delta}}\Bigr)E_1(t)\le Ce^{-\alpha^*t}E_1(0).
		\end{align*}
		Now set $\bar\lambda=\min\bigl\{\frac{2k\delta}{1+k^2\delta},\sqrt{\delta}\bigr\}$; then
		\begin{align*}
			E_1(t)\le Ce^{-\alpha^*t}E_1(0),
		\end{align*}
		or equivalently,
		\begin{align*}
			\|\bar{u}_x\|^2+\|\bar{u}_t\|^2\le C\bigl(\|\bar{u}_{0x}\|^2+\|\bar{u}_1\|^2\bigr)e^{-\alpha^*t}.
		\end{align*}
		Moreover,
		\begin{align*}
			\|\bar{u}\|^2\le\|\bar{u}_0\|^2+\int_0^t\|\bar{u}_s(\cdot,s)\|^2ds\le \|\bar{u}_0\|^2+C\bigl(\|\bar{u}_{0x}\|^2+\|\bar{u}_1\|^2\bigr).
		\end{align*}
		
		To complete the estimate, we also need higher‑order a priori estimate. As in the argument in Section \ref{sec3.2}, we consider the higher‑order energy functionals
		\begin{align*}
			E_2(t)&:=\frac12\int_0^1\bigl(\bar{V}_t^2+h(\bar{u}_x)\bar{V}_x^2\bigr)dx,\\
			E_3(t)&:=\frac12\int_0^1\bigl(\bar{W}_t^2+h(\bar{u}_x)\bar{W}_x^2\bigr)dx,
		\end{align*}
		and the corresponding modified energies
		\begin{align*}
			F_2(t)&:=\bar\lambda'\int_0^1x\bar{V}_x\bar{V}_tdx,\\
			F_3(t)&:=\bar\lambda''\int_0^1x\bar{W}_x\bar{W}_tdx,
		\end{align*}
		where $\bar{V}:=\bar{u}_t$ and $\bar{W}:=\bar{V}_t=\bar{u}_{tt}$. Repeating the arguments of Lemmas \ref{le3.2}--\ref{le3.3} word by word with $1/\tau=0$ and all the $f$-terms dropped (so that the $1/\tau$-dissipation and the $f$-source terms vanish), we obtain the analogous estimates for $E_2$ and $E_3$.
		
		Combining the first-order bound above with these higher-order estimates and recovering the pure spatial derivatives from the equation $\bar u_{tt}=h(\bar u_x)\bar u_{xx}$ (cf.\ \eqref{71}--\eqref{73} with $1/\tau=0$), we obtain
		\[\|\bar u_{xt}\|+\|\bar u_{tt}\| \le Ce^{-\bar\gamma t}\big(\|\bar u_{0x}\|_1+\|\bar u_1\|_1\big),\]
		\[\|\bar u_{xtt}\|+\|\bar u_{ttt}\|+\|\bar u_{xx}\|+\|\bar u_{xxt}\|+\|\bar u_{xxx}\|\le Ce^{-\bar\gamma t}\big(\|\bar u_{0x}\|_2+\|\bar u_1\|_2\big),\]
		for some $\bar\gamma>0$ (any $\bar\gamma$ smaller than the rate in the first-order estimate). This completes the proof of Proposition \ref{A.pro}.
	\end{proof}

	\begin{proof}[Proof of Theorem \ref{thm2}]
		It suffices to prove the decay estimate \eqref{bar u*}. By Proposition \ref{A.pro}, we have
		\begin{align}\label{A-uxut}
			\|\bar{u}_x\|+\|\bar{u}_t\|\le Ce^{-\bar{\alpha}t}(\|\bar{u}_{0x}\|+\|\bar{u}_1\|).
		\end{align}
		Following the same argument as in the proof of Theorem \ref{thm3}, we define 
		$\bar{u}_*:=\lim\limits_{t\to\infty}\int_0^1\bar{u}dx$, which satisfies
		\begin{align*}
			\left\|\int_0^1\bar{u}dx-\bar{u}_*\right\|\le Ce^{-\bar{\alpha}t}(\|\bar{u}_{0x}\|+\|\bar{u}_1\|).
		\end{align*}
		Then, by Poincar\'e's inequality, we obtain
		\begin{align*}
			\|\bar{u}-\bar{u}_*\|\le \left\|\bar{u}-\int_0^1\bar{u}dx\right\|+\left\|\int_0^1\bar{u}dx-\bar{u^*}\right\|\le Ce^{-\bar{\alpha}t}(\|\bar{u}_{0x}\|+\|\bar{u}_1\|).
		\end{align*}
		
		Finally, for the component $\bar{v}$, combining \eqref{A-uxut} with \eqref{2a} yields
		\begin{align*}
			\|\bar{v}_x\|+\|\bar{v}_t\|\le Ce^{-\bar{\alpha}t}(\|\bar{u}_{0x}\|+\|\bar{u}_1\|).
		\end{align*}
		Similarly, there exists a constant $\bar{v}_*$ such that
		\begin{align*}
			\|\bar{v}-\bar{v}_*\|\le Ce^{-\bar{\alpha}t}(\|\bar{u}_{0x}\|+\|\bar{u}_1\|).
		\end{align*}
		This completes the proof.
	\end{proof}
	
	\medskip
	
	{\bf Acknowledgement} The research of Z. Hu was supported by the Graduate Innovation
	Fund of Jiangxi Province (No. YC2025-B066). The research of S. Li was partially supported by Early-Career Young Scientists and Technologists Project of Jiangxi Province (No.20262BEJ730378) and Science and Technology Research Project of Jiangxi Provincial Department of Education (No.GJJ2500210). The research of M. Mei was partially supported by National Natural Science Foundation of China
(W2431005) and Natural Sciences and Engineering Research
Council of Canada (NSERC RGPIN 2022-03374). The research of Z. Wang was partially supported by the National Natural Science Foundation of China(No.12261047), and Jiangxi Provincial Natural Science Foundation (No.20243BCE51015).

\

{\bf Data Availability Statement} This article has no associated data.

\

{\bf Declarations}

\

{\bf Conﬂict of interest} The authors declare that there are no conﬂicts of interest.

\

{\bf Ethics approval and consent to participate} Not applicable.

\

{\bf Consent for publication} Not applicable.

\

{\bf Materials availability} Not applicable.

\

{\bf Code availability} Not applicable.

\end{document}